\documentclass{article}

\usepackage[english]{babel}
\usepackage[utf8]{inputenc}
\usepackage[T1]{fontenc}

\usepackage{authblk}
\usepackage{geometry}
\usepackage[numbers,sort&compress]{natbib}

\usepackage{rotating}

\usepackage{dsfont}
\usepackage{amssymb}
\usepackage{amsthm}
\usepackage{amsmath} 
\usepackage{enumerate} 
\usepackage{mathtools}
\usepackage{siunitx} 
\usepackage{gensymb}   
\usepackage{pgfplots}   
\pgfplotsset{compat=newest} 
\usepackage{xcolor} 
\usepackage{tikz}
\usetikzlibrary{plotmarks,backgrounds,patterns,calc,angles,quotes}
\usetikzlibrary{external}
\usepgfplotslibrary{groupplots}
\tikzset{external/only named=true}
 \tikzexternaldisable 
    
\newcommand*{\pgfmathsetnewmacro}[2]{%
    \newcommand*{#1}{}
    \pgfmathsetmacro{#1}{#2}%
}%
 
\usepackage{booktabs}

\usepackage{todonotes}
\makeatletter
\renewcommand{\todo}[2][]{\tikzexternaldisable\@todo[#1]{#2}\tikzexternalenable}
\makeatother

\usepackage{hyperref}

\usepackage[noabbrev,capitalize]{cleveref}
\crefname{equation}{}{}
\Crefname{equation}{}{}

\newtheorem{theorem}{Theorem}
\theoremstyle{plain}

\newtheorem{lemma}[theorem]{Lemma}
\theoremstyle{plain}

\newtheorem{corollary}[theorem]{Corollary}
\theoremstyle{plain}

\newtheorem{remark}[theorem]{Remark}
\theoremstyle{plain}

\newtheorem{assumption}[theorem]{Assumption}
\theoremstyle{plain}

\newtheorem{notation}[theorem]{Notation}
\theoremstyle{plain}

\DeclareMathOperator{\divergence}{div}

\DeclareMathOperator{\dx}{\,dx}
\DeclareMathOperator{\dt}{\,dt}

\DeclareMathOperator{\dz}{\,dz}

\newcommand{\R}{\mathbb{R}} 
\newcommand{\N}{\mathbb{N}}

\title{Optimal Control of Sliding Droplets using the Contact Angle Distribution
\footnote{The first author acknowledges the German Research Foundation (DFG) for the financial support within the project RE 1705/16-1.
}}

\author[1]{Henning Bonart}
\author[2,*]{Christian Kahle}

\affil[1]{Technische Universit\"at Berlin, Process Dynamics and Operations Group,\protect\\Straße des 17. Juni 135, 10623 Berlin, Germany}
\affil[2]{Universität Koblenz-Landau, Campus Koblenz,\protect\\ Universit{\"a}tsstra{\ss}e 1, 56070 Koblenz, Germany}
\affil[*]{Corresponding author: \url{kahle@uni-koblenz.de} (Christian Kahle)}

\date{\today}

\begin{document}

\maketitle

\begin{abstract}
	Controlling the shape and position of moving and pinned droplets on a solid surface is an important feature often found in microfluidic applications.
In this work, we consider a well investigated phase field model including contact line dynamics as the state system for an (open-loop) optimal control problem. 
Here the spatially and temporally changeable contact angles between droplet and solid are considered as the control variables.
We consider a suitable, energy stable, time discrete version of the state equation in our optimal control problem.
We discuss regularity of the solution to the time discrete state equation and its continuity and differentiability properties.
Furthermore, we show existence of solutions and state first order optimality conditions to the optimal control problem. 
We illustrate our results by actively pushing a droplet uphill against gravity in an optimal way.

\end{abstract}

\section{Introduction}
Controlling the shape and position of moving and pinned droplets on a solid surface is an important feature often found in microfluidics applications.
On a lab-on-a-chip droplets can be transported across the solid surface by a contact angle gradient, and merged, split or mixed in a controlled fashion~\cite{Pollack2002}. 
Thereby, the shape of the droplets influences the heat and mass exchanged with the solid surface and the surrounding fluid phase~\cite{Al-Sharafi2018}.
Furthermore, the distribution of nutritiens and the direction of growth of bio-films and cell cultures depends on the shape and the surface structure, too~\cite{Epstein2011}.
In optical applications, liquid droplets can act as flexible lenses with continuous refraction index ranges. 
The curvature and hence their focal length can be tuned by adjusting their shape through the contact angle distribution~\cite{Hou2007}.

In all these processes and applications the shape and position of the droplet (or gas-liquid interface) plays a significant role.
However, automating them, e.g., for high-throughput applications, does rarely involve model-based optimal control strategies.
The potential of optimal control for microfluidics is for example shown in the following publications.
In~\cite{Laurain2015} the control of the footprint and shape of a static droplet is presented.
The position of a moving droplet and its shape without the influence of gravity is considered in~\cite{Antil2017}.
Results on  the position of the gas-liquid interface of rising liquid in a capillary are provided in~\cite{Fumagalli2017}.
 
In this work, we are concerned with the optimal control of droplets where the static contact angle between solid surface and droplet serves as the control variable.
In~\Cref{fig:setup_intro}, a general physical setting of the problem is illustrated.
Initially, a liquid droplet is placed on an inclined solid surface (dark gray).
Due to gravity $g$, the droplet slides down the surface.
However, with the help of the patches $u_1$ to $u_4$ we control the contact angles $\theta_1$ and $\theta_2$ between droplet and solid.
Depending on the actual objective, we use this control e.g. to track a desired shape (dashed) over a given time horizon.
In this way, we are able to impose the desired shape and position at a specific time (light gray).
As observed in the very famous experiments by~\citet{Chaudhury1992}, we are even able to push the droplet uphill against gravity.
Note, that in practical applications, the control patches represent electrodes and the contact angle is varied using an electric potential (so-called electrowetting, see~\cite{Mugele2018}).
For details on the physical background as well as the technical implementation in devices we refer to ~\cite{Mugele2018}.
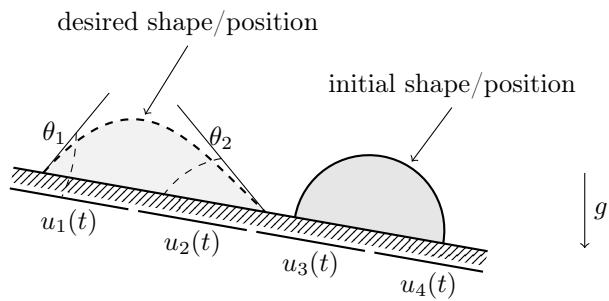
\begin{figure}[h]
	\centering
	\tikzsetnextfilename{Setup}
\begin{tikzpicture}[scale=4.0]
	\filldraw[rotate=-10,thick, dashed, fill=black!05] (0.0,0) node[](r1){} .. controls (0.25,0.4) and (0.45,0.2) .. (0.75,0) node(r2){};
	\draw[<-, rotate=-10] (0.3, 0.25) -- (0.4, 0.5) node[above, align=center]{desired shape/position};
	%
	\filldraw[thick, rotate=-10, fill=black!10] (0.85, 0) node(r1){} arc(180:0:0.25)node(r2){};
	\draw[<-, rotate=-10] (1.2, 0.25) -- (1.3, 0.45) node[above, align=center]{initial shape/position};
	%
	\draw[rotate=-10, thick] (-0.1,0) -- (1.5,0) node(b){};
	\fill[rotate=-10,thick,pattern=north east lines] (-0.1,0) -- (1.5,0) -- ++(0,-0.05) -- (-0.1,-0.05) -- ++(0, 0.05);
	%
	\draw[rotate=-10,thick] (-0.1, -0.07) -- node[midway,below] {$u_1(t)$}++(0.4, 0) node (r){};
	\draw[rotate=-10,thick] (r) -- node[midway,below] {$u_2(t)$}++(0.4, 0) node (r){};
	\draw[rotate=-10,thick] (r) -- node[midway,below] {$u_3(t)$}++(0.4, 0) node (r){};
	\draw[rotate=-10,thick] (r) -- node[midway,below] {$u_4(t)$}++(0.4, 0) node (r){};
	%
	\draw[->] (1.8,0)-- node[right]{$g$}++(0.0,-0.25);
	%
	%
	\draw[rotate=-10] (0.0,0) --node[midway](c1){} ++(0.17,0.3);
	\draw[dashed] (c1) arc(0:-25:0.5);
	\node[left] at (c1) {$\theta_1$};
	\draw[rotate=-10] (0.75,0) --node[midway](c2){} ++(-0.35,0.3);
	\draw[dashed] (c2) arc(100:145:0.3);
	\node[above] at (c2) {$\theta_2$};
\end{tikzpicture}
	\caption{Physical setting of the optimal control problem.} 
	\label{fig:setup_intro} 
\end{figure} 

This paper is organized as follows. In Section~\ref{sec:M} we introduce the model for the moving contact line dynamics and review 
some properties. 
Thereafter we introduce a time discrete approximation of this model in Section~\ref{sec:S} and 
investigate the regularity properties of the resulting equations. Moreover, we show continuity and differentiability
properties that are necessary in the subsequent section. 
In Section~\ref{sec:O} we introduce an optimization problem, that models the control of 
a droplet by the contact angle distribution. Using the results from Section~\ref{sec:S} we show
existence of solutions and derive first order optimality conditions.
Finally, in Section~\ref{sec:N} we illustrate our results by optimally pushing a droplet uphill.

\section{The model for moving contact line dynamics}
\label{sec:M}
In this section we introduce the mathematical model under consideration.
It is a Cahn--Hilliard Navier--Stokes system 
with a moving contact line model for the Cahn--Hilliard system
and no-slip boundary data for the Navier--Stokes system, 
 see for example~\cite{AbelsGarckeGruen-CHNSmodell,
AbelsBreit_weakSolution_nonNewtonian_DifferentDensities, 
2006-QianWangShen-Variational-MovingContactLine--BoundaryConditions,
GruenMetzger__CHNS_decoupled}. 
 
The model consists of a velocity field $v(t)$ and a pressure field $p(t)$ that satisfy the Navier--Stokes-type equation \eqref{eq:M:1_NS1}--\eqref{eq:M:2_NS2}
and a phase field $\varphi(t)$ and a chemical potential $\mu(t)$ that satisfy the advective Cahn--Hilliard equation \eqref{eq:M:3_CH1}--\eqref{eq:M:4_CH2}.
In continuous and strong setting the model reads as follows.

Let $\Omega \subset \R^n$, $n\in \{2,3\}$ denote 
a polynomially/polygonally bounded Lipschitz domain with boundary $\partial\Omega$  and with unit outer normal $\nu_\Omega$.
Let $(0,T)$ denote a time interval.
Given sufficient smooth initial data $v_0$, and $\varphi_0$ find $v(t,x)$, $p(t,x)$, $\varphi(t,x)$, and $\mu(t,x)$  such that 
for almost all $t\in (0,I]$ it holds
\begin{align}
\rho\partial_t v + ((\rho v + J)\cdot\nabla) v + R\frac{v}{2}
-\mbox{div}\left(2\eta Dv\right) + \nabla p &= -\varphi\nabla \mu + \rho g
&& \mbox{ in } \Omega, \label{eq:M:1_NS1}\\
-\mbox{div}(v) &= 0 
&& \mbox{ in } \Omega,\label{eq:M:2_NS2}\\
\partial_t \varphi + v \cdot\nabla \varphi - b\Delta \mu &= 0 
&& \mbox{ in } \Omega,\label{eq:M:3_CH1}\\
-c_W\sigma_{lg}\epsilon\Delta \varphi + c_W\sigma_{lg}\epsilon^{-1}W^\prime(\varphi) &= \mu 
&& \mbox{ in } \Omega,\label{eq:M:4_CH2}\\
v &= 0 
&& \mbox{ on } \partial\Omega, \label{eq:M:5_NS_BC_1}\\
r\partial_t\varphi + L(\varphi)&=0 
&& \mbox{ on } \partial\Omega,\label{eq:M:7_CH_BC}\\
\nabla \mu \cdot\nu_\Omega &= 0 
&&\mbox{ on } \partial\Omega,
\label{eq:M:8_mu_neumann}
\end{align}
where we set 
$J := -b\frac{d\rho}{d\varphi}\nabla \mu$,
$R := -b\nabla \frac{d\rho}{d\varphi}\cdot\nabla\mu$,
$2Dv := \nabla v + (\nabla v)^t$,
and
$ L:= c_W\sigma_{lg}\epsilon \nabla\varphi\cdot\nu_\Omega + \gamma_u^\prime(\varphi)$.

The (nonlinear) functions $\rho(\varphi)$ and $\eta(\varphi)$ denote the density  and the viscosity of the fluid, respectively.
See Remark~\ref{rm:M:density} for further discussion of $\rho$ and $\eta$. 
We note, that the additional term $R$ in \eqref{eq:M:1_NS1} appears from the nonlinearity of $\rho$ and vanishes in case of a linear function $\rho(\varphi)$,
see \cite{AbelsBreit_weakSolution_nonNewtonian_DifferentDensities}.
The gravitational acceleration is denoted by $g$, 
while $b>0$ denotes the mobility of the fluid, that for simplicity is taken as constant.
The constant $\sigma_{lg}$ denotes the surface tension at the fluidic interface between the two phases.
The interface is considered as diffuse with a width proportional to $\epsilon$.
The function $W(\varphi)$ denotes a dimensionless potential
of double-well type, with two strict minima at
$\pm 1$ that define the pure phases. See Remark~\ref{rm:M:Potential} for further discussion of $W$.
The constant $c_W$ is also defined in Remark~\ref{rm:M:Potential}.
In the following it is convenient to call the phase that is defined by $\varphi \equiv -1$ as \textit{gas} 
and the phase that is defined by $\varphi \equiv 1$ as \textit{liquid}.
The material outside of $\Omega$ is called \textit{solid}.
Finally, $\gamma_u(\varphi)$ denotes the contact line energy and is further explained in Remark~\ref{rm:M:ContactEnergy}. 
The constant $r\geq0$ denotes a phenomenological parameter allowing for nonequilibrium contact angles at the contact line.
 
\bigskip

Existence of a solution to  \eqref{eq:M:1_NS1}--\eqref{eq:M:8_mu_neumann} without the additional term  $R$ is shown in
\cite{GruenMetzger__CHNS_decoupled}  and several numerical schemes are tested.
Concerning further analytical results for the bulk model with homogeneous 
boundary data we refer to
 \cite{AbelsGarckeGruen-CHNSmodell,
 AbelsDepnerGarcke_CHNS_AGG_exSol,
AbelsDepnerGarcke_CHNS_AGG_exSol_degMob,
AbelsBreit_weakSolution_nonNewtonian_DifferentDensities,
Gruen_convergence_stable_scheme_CHNS_AGG}.
For results on phase field models,  that contain a contact line model, we refer to
\cite{
2016-GalGrasselliMiranville-CHNSMCL-EqualDensityExSol,
2017_ColliGilardiSprekels_CH_with_dynamicBoundary,
2018-XuDiHu-SharpInterfaceLimit-NavierSlipBoundary}.
 
Concerning numerical schemes for \eqref{eq:M:1_NS1}--\eqref{eq:M:8_mu_neumann} and submodels thereof, we refer to
\cite{Tierra_Splitting_CHNS,
GonzalesTierra_linearSchemes_CH,
GruenMetzger__CHNS_decoupled,
Gruen_Klingbeil_CHNS_AGG_numeric,
2016-GarHK_CHNS_AGG_linearStableTimeDisc,
Aland__time_integration_for_diffuse_interface,
AlandChen__MovingContactLine,
YuYang_MovingContactLine_diffDensities,
2014-WuZwietenZee-StabilizedSecondOderConvecSplittingCHmodels,
2012-GaoWang-gradientStableSchemePhaseFieldMCL,
2011-HeGlowinskiWang-LeastSquaresCHNSMCL,
2015-ShenYangYu-EnergyStableSchemesForCHMCL-Stabilization,
2017-YangJu-InvariantEnergyQuadratization,
2018-ShenXuYang-SAV_for_gradientFlows,
2019-ChengPromislowWetton-AsymptoticBehaviourTiemSteppingPhaseFields}.

For results on control and boundary control of the Cahn--Hilliard and Cahn--Hilliard
Navier--Stokes equation we refer to
\cite{2019-ColliSignori-BoundaryControlCahnHilliard_dynamicBoundary,
GarHK_optContr_twoPhaseFlow,
HintHKK_CHNSOPT_DWR,
Hintermueller_Keil_Wegner__OPT_CHNS_density,
FrigeriGrasselliSprekels_optimalControlNonlocalCHNSdegenerateMobility,
2019_EbenbeckKnopf_OPT_tumor,
GraeHHK_SimContr_nonsmooth_CHNS,
HintK_opt_geom_PDE}.  
  
\begin{remark}[Nonlinear density and viscosity]
  \label{rm:M:density}
Note that in general there is no quantitative upper bound available for $\varphi$ and thus 
in particular $|\varphi| > 1$ is commonly observed.
Thus a linear relation between $\varphi$ and $\rho$ might lead to negative densities and viscosities in practice.
This especially appears for large density ratios, compare, e.g.,\ \cite[Rem. 4.1]{GruenMetzger__CHNS_decoupled}.
Therefore we use the modification proposed in \cite{AbelsBreit_weakSolution_nonNewtonian_DifferentDensities}
and add the term $R$ in \eqref{eq:M:1_NS1} that compensates for a nonlinear density.
We also refer to, e.g.,\ \cite[Rem. 2.1]{Gruen_convergence_stable_scheme_CHNS_AGG}, 
\cite[Rem. 1]{Bonart2019b}, and
\cite[Rem. 6]{2016-GarHK_CHNS_AGG_linearStableTimeDisc} for further discussions of this topic.
In the following we assume that $\rho$ and $\eta$ are bounded, strictly positive, continuously differentiable and
globally Lipschitz continuous. Additionally we require $\rho \in C^2(\R)$ and define $\rho_{\min} := \min_{t\in \R} \rho(t)$.
Moreover $\rho(1) = \rho_l$, $\rho(-1) = \rho_g$, $\eta(1) = \eta_l$, and $\eta(-1) = \eta_g$ is satisfied, 
where $\rho_{l/g}$ denotes the density of the gas and liquid phase, while $\eta_{l/g} $ denote the corresponding viscosity.

We note, that with a nonlinear density function, while $\int_\Omega \varphi\dx$ is a conserved quantity, 
the total mass $\int_\Omega \rho(\varphi)\dx$ is only conserved if $\rho(\varphi)$ is a linear 
function on the (a-priori unknown) image of $\varphi$, 
see e.g., \cite[Rem. 1]{2016-GarHK_CHNS_AGG_linearStableTimeDisc}. 

\end{remark}

\begin{remark}[The free energy potential]
  \label{rm:M:Potential}
  The free energy potential $W$ is a function with exactly two minima at $\pm 1$ with $W(\pm1) = 0$.
  Additonally we assume, that $W\in C^{2,1}(\R)$ 
  and that there exists a constant $C>0$ such that
\begin{align}
\label{ass:M:W_poly_bounded}
  W(\varphi) \leq C(1+|\varphi|^4), \quad
  |W^\prime(\varphi)| \leq C (1+|\varphi|^3),\quad
|W^{\prime\prime}(\varphi)| \leq C (1+|\varphi|^2).
  \end{align}
The function $W$ further admitts a convex-concave splitting $W = W_+ + W_-$ such that
$W_+$ is a convex function with $W_+^\prime(0) = 0$ 
and $W_-$ is a concave function.
The assumption \eqref{ass:M:W_poly_bounded} also holds for $W_+$ and $W_-$ individually.

The parameter $c_W$ is defined as 
$c_W^{-1} := \int_{-\infty}^\infty 2 W(\Phi_0(z))\dz$, where $\Phi_0(z)$ 
denotes the first order approximation of $\varphi$ that satisfies
$\Phi_0(z)_{zz} = W^\prime(\Phi_0(z))$
see \cite[Sec. 4.3.4]{AbelsGarckeGruen-CHNSmodell}.

We refer to, e.g.,  \cite[Rem. 2]{Bonart2019b} 
for  further discussion of free energy potentials.

\end{remark}

\begin{figure}
  \centering
\tikzsetnextfilename{YoungsLaw}

\begin{tikzpicture}[x=1cm, y=1cm]

\pgfmathsetnewmacro{\LL}{-4}
\pgfmathsetnewmacro{\RR}{ 4}

\pgfmathsetnewmacro{\rr}{1.0}

\draw (\LL,0) -- (\RR,0);

\draw (\LL*0.40,-0.1) node[anchor = south]{$\sigma_{sg}$}; 
\draw (\LL,0.0) node[anchor=south west]{gas};
\draw (\LL,0.30) node[anchor=south west]{($\varphi \equiv -1$)};
\draw (\RR*0.25,-0.1) node[anchor = south]{$\sigma_{sl}$};
\draw (\RR,0.0) node[anchor=south east]{liquid};
\draw (\RR,0.30) node[anchor=south east]{($\varphi \equiv 1$)};
\draw (0.15*\RR,0.70) node[anchor= south]{$\sigma_{gl}$}; 
\draw (\LL,-0.25) node[anchor=north west]{solid};
\fill[pattern=north east lines] (\LL,-0.25) rectangle (\RR,0);

\draw (-1.0,0) .. controls (1.0,1.0) .. (\RR,1.2);

\draw (0.5,0) arc (0:36:\rr);
\draw (0.0,-0.1) node[anchor = south]{$\theta_{eq}$};

\draw (0.0,-0.5) node[anchor=north]{$\sigma_{sg} - \sigma_{sl} = \sigma_{gl}\cos \theta_{eq}$};

\end{tikzpicture}
\caption{Definition of parameters in Young's law.}
\label{fig:M:YoungsLaw}
\end{figure}
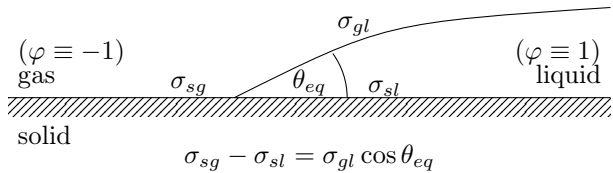

\begin{remark}[The contact line energy]
  \label{rm:M:ContactEnergy}
 The contact line energy $\gamma_u$ is given by
 \begin{align*}
   \gamma_u(\varphi) := \frac{\sigma_{sl}+\sigma_{sg}}{2} + \sigma_{lg}
   \left( \cos(\theta_{eq}) + Bu \right)
   \vartheta(\varphi).
 \end{align*}
 Here $\sigma_{sl}$ denotes the physical surface tension between solid and liquid,
 $\sigma_{sg}$ denotes the physical surface tension between solid and gas,
 and $\sigma_{lg}$, as before,  denotes the physical surface tension between liquid and gas.
 The equilibrium contact angle $\theta_{eq}$ between the interface and the solid is measured in 
 the liquid phase. 
 These variables satisfy Young's law, namely
 \begin{align*}
   \sigma_{sl} - \sigma_{sg} = \sigma_{lg}\cos{\theta_{eq}},
 \end{align*}
 see Figure~\ref{fig:M:YoungsLaw}.

 The function $\vartheta$ satisfies $\vartheta\in C^{1,1}(\R) \cap C^2(\R)$	
 and there exists $L < \infty$ such that 
 $\max_{s\in \R} |\gamma^{\prime\prime}(s)| \leq L$. 
 Further it holds $\vartheta(t) = -\frac{1}{2}$ for $t \leq -1$ and $\vartheta(t) = \frac{1}{2}$ for $t\geq 1$.
  A typical choice is $\vartheta(\varphi) = \frac{1}{2}\sin(\frac{\pi}{2}\min(1,\max(-1,\varphi)))$. 
 We refer to \cite[Rem. 3]{Bonart2019b} and the references therein for a discussion.
 
 The function $u\in U$ denotes a given control from some control space $U$, while
 $B:U \to L^2(0,T;L^2(\partial\Omega))$ denotes a linear, continuous and injective control operator.
 The contact line model implies  natural bounds for $Bu$
 and we thus define the convex and closed subset 
 \begin{align}
   U_{ad}^0 := \{u\in U \,|\, -1 \leq \cos(\theta_{eq}) + Bu \leq 1 \} \subset U
 \end{align}
 as largest set of admissable controls. 
Note that, since $B$ is injective, in fact $U_{ad}^0$ is a bounded subset.

\end{remark}

\begin{notation}
We use the common notation for Sobolev and Hilbert spaces, see e.g. \cite{Adams_SobolevSpaces}.
Especially by $L^2(\Omega)$ we denote the Hilbert space of square integrable functions on a
domain $\Omega$. The norm and inner product in $L^2(\Omega)$ are denoted by $\|\cdot\|$ and  $(\cdot,\cdot)$. 
The space $L^\infty(\Omega)$ denotes the space of essentially bounded functions, and
the space $H^1(\Omega)$ denotes the space of square integrable functions, that admit a weak derivative
that is square integrable.
 
Additionally to the standard notation, we introduce the space of weakly solenoidal functions
\begin{align*}
  H_\sigma(\Omega) := \{ v\in H^1_0(\Omega)^n\,|\, (\divergence{v},q) = 0 \, \forall q \in L^2(\Omega),\, \int_\Omega q \dx = 0\}.
\end{align*}

Moreover, in the following we restrict to solenoidal velocity fields from $H_\sigma(\Omega)$ and thus skip the pressure in the equations.
 
For notational convenience in the following we set $\sigma = c_W \sigma_{lg}$.

\end{notation}

\section{Analysis of the time discrete model}
\label{sec:S}
 
For a practical implementation we introduce a time
grid $0 = t_0 < t_1 < \ldots<t_{m-1} < t_m< \ldots <t_M = T$.
For the sake of notational simplicity let the time grid be equidistant with step size
$\tau>0$.
We consider the following time discrete version of \eqref{eq:M:1_NS1}--\eqref{eq:M:8_mu_neumann} in weak form.

Given $\varphi_0\in H^1(\Omega) \cap L^\infty(\Omega)$, $v_0 \in H_\sigma(\Omega)$, and $u\in U_{ad}^0$. 
Find sequences 
$v_{\tau} = (v^m)_{m=1}^M \in (H_\sigma(\Omega))^M$,
$\varphi_{\tau} = (\varphi^m)_{m=1}^M \in (H^1(\Omega)\cap L^\infty(\Omega))^M$, 
$\mu_{\tau} = (\mu^m)_{m=1}^M \in (W^{1,q}(\Omega))^M$, $q>n$,
such that for $m=1,\ldots,M$
 and for all 
$w^m \in H_\sigma(\Omega)$,
$\Phi^m \in H^1(\Omega)$, and
$\Psi^m \in H^1(\Omega)$ 
the following equations 
\begin{align}
   \frac{1}{\tau}\left( \frac{\rho^m+\rho^{m-1}}{2} v^m -\rho^{m-1}v^{m-1},w^m\right)
  + a(\rho^{m-1}v^{m-1} + J^{m},v^m,w^m)
  + (2\eta^{m}Dv^m,Dw^m)\nonumber\\
		+( \varphi^{m-1}\nabla \mu^m,w^m) 
		-(g\rho^{m},w^m) &= 0, \label{eq:S:1_NS}\\
   \frac{1}{\tau}(\varphi^{m} - \varphi^{m-1},\Psi^m)
   -(\varphi^{m-1}v^{m-1},\nabla \Psi^m) + \frac{\tau}{\rho_{\min}}(|\varphi^{m-1}|^2 \nabla \mu^m,\nabla \Psi^m)
		+b(\nabla \mu^m,\nabla \Psi^m) &= 0 \label{eq:S:3_CH1},\\
  \sigma\epsilon(\nabla \varphi^m,\nabla \Phi^m)+\frac{\sigma}{ \epsilon}(W_+^\prime(\varphi^m) + W_-^\prime(\varphi^{m-1}),\Phi^m)  
   - (\mu^m,\Phi^m) \nonumber\\
   +r\left(B^m,\Phi^m\right)_{\partial\Omega} 
    +\left( \frac{S_\gamma}{2} (\varphi^m-\varphi^{m-1}) + \gamma_u^\prime(\varphi^{m-1}),
    \Phi^m\right)_{\partial\Omega} 
    &= 0,
    \label{eq:S:4_CH2}  
\end{align}
with
$J^{m} := -b\rho^\prime(\varphi^{m})\nabla \mu^{m}$,
$B^m:= \left(\frac{\varphi^m-\varphi^{m-1}}{\tau}\right)$,
$\rho^{m} := \rho(\varphi^{m})$, 
$\rho^{m-1} := \rho(\varphi^{m-1})$, 
$\eta^{m} := \eta(\varphi^{m})$, 
and $\eta^{m-1} := \eta(\varphi^{m-1})$.
For some $q>n$ (if $n=2$) or $q\geq n$ (if $n=3$) the trilinear form $a:(L^q(\Omega))^n\times H^1(\Omega)^n \times H^1(\Omega)^n$ is defined by
\begin{align*}
a(u,v,w) := \frac{1}{2}\int_\Omega ((u\cdot \nabla )v )\cdot w - \frac{1}{2}\int_\Omega ((u\cdot \nabla) w )\cdot v.
\end{align*}
The contact line energy $\gamma^\prime_u(\varphi^{m-1})$ is given by
$\gamma^\prime_u(\varphi^{m-1}) = \sigma_{lg}(\cos(\theta_{eq}) + B_mu)\vartheta^\prime(\varphi^{m-1})$,
with 
$B_mu:= \frac{1}{\tau}\int_{t_{m-1}}^{t_m} (Bu)(t)\dt$.
The parameter $S_\gamma$ is a stabilization parameter and satisfies $S_\gamma \geq \frac{1}{2}\max_s|\gamma^{\prime\prime}(s)|$, see e.g.
\cite{Bonart2019b,AlandChen__MovingContactLine}.

We note, that if necessary we consider $v_\tau$, $\varphi_\tau$, and $\mu_\tau$ as piecewise constant functions in time to evaluate time depending norms. Here we set $\varphi_\tau(t) \equiv \varphi^m$ for $t\in (t_{m-1},t_m]$,
and correspondingly for $v_\tau$ and $\mu_\tau$. 

\begin{assumption}
The regularity $\mu^m \in W^{1,q}(\Omega)$ with $q>n$ is required for the trilinear form $a(\cdot,\cdot,\cdot)$ to be well defined,
since $\nabla \mu^m$ appears in the first argument of the trilinear form $a$. 
For $n=3$ this can be reduced to $\mu \in W^{1,3}(\Omega)$
and in the following we assume $n=3$.
\end{assumption}
  
\begin{remark}[On the time discretizations of $W$ and $\gamma$] 
 For the time discretization of $\gamma$ we use a linear stabilization scheme, to guarantee energy stability of the scheme, see Theorem~\ref{thm:S:enerInequ}.
 However, in the bulk domain, we use a convex-concave splitting for $W$ to guarantee energy stability.
 
 We note, that the proposed linearization scheme for $\gamma_u$ might also be applied for $W$. 
 Typically we obtain broader interfaces and a severe effect
 on the interface dynamic with this approach and therefor refrain from using this scheme for the bulk energy.
  
 In \cite{Bonart2019b} we use the convex-concave splitting scheme also for the contact line energy $\gamma_u$.
 In Section~\ref{sec:O} we consider the term $B_mu$ in $\gamma_u$ as control variable 
 and thus $\gamma_u^\prime$ in general does not have a uniform sign. Therefor we can not use 
 a similar convex-concave splitting for the discretization of $\gamma_u$ in this work.
 
 For further discretization schemes, we also refer to 
 \cite{2014-WuZwietenZee-StabilizedSecondOderConvecSplittingCHmodels, 
 2017-YangJu-InvariantEnergyQuadratization, 
 2018-ShenXuYang-SAV_for_gradientFlows, 
 2011-HeGlowinskiWang-LeastSquaresCHNSMCL,
 Bonart2019b}
 and note, that typically the convex-concave splitting seems to be a very good compromise
 between stable interfaces and the validity of an energy inequality in the time discrete setting.
\end{remark}

\begin{theorem}[Energy inequality]
\label{thm:S:enerInequ}
Assume, that there exists a solution to \eqref{eq:S:1_NS}--\eqref{eq:S:4_CH2}. Then the following
energy inequality holds on one time instance $m$.
\begin{equation}
  \label{eq:S:enerIneqOneStep}
  \begin{aligned}
    &\frac{1}{2}\int_\Omega \rho^m|v^m|^2 + \sigma\int_\Omega \frac{\epsilon}{2}|\nabla \varphi^m|^2 + \frac{1}{\epsilon}W(\varphi^m)
    + \int_{\partial\Omega}\gamma(\varphi^m)\\
    &+ \tau \left(\int_\Omega 2\eta^{m} |Dv^m|^2 + b\int_\Omega |\nabla \mu^m|^2 + r \int_{\partial\Omega} |B^m|^2 \right)\\
   &\leq
    \frac{1}{2}\int_\Omega \rho^{m-1}|v^{m-1}|^2 + \sigma\int_\Omega \frac{\epsilon}{2}|\nabla \varphi^{m-1}|^2 + \frac{1}{\epsilon}W(\varphi^{m-1})
    + \int_{\partial\Omega}\gamma(\varphi^{m-1}) + \tau\int_\Omega \rho^{m} g\cdot v^m.
  \end{aligned} 
\end{equation}
Moreover, by summing over \eqref{eq:S:enerIneqOneStep}  we obtain
\begin{equation}
  \label{eq:S:enerIneqSummed}
  \begin{aligned}
    &\frac{1}{2}\int_\Omega \rho^m|v^m|^2 + \sigma\int_\Omega \frac{\epsilon}{2}|\nabla \varphi^m|^2 + \frac{1}{\epsilon}W(\varphi^m)
    + \int_{\partial\Omega}\gamma(\varphi^m)\\
    & + \tau \sum_{k=1}^m
    \left(\int_\Omega 2\eta^{k} |Dv^k|^2 + b\int_\Omega |\nabla \mu^k|^2 + r \int_{\partial\Omega} |B^k|^2 \right)\\
    &\leq
    \frac{1}{2}\int_\Omega \rho^{0}|v^{0}|^2 + \sigma\int_\Omega \frac{\epsilon}{2}|\nabla \varphi^{0}|^2 + \frac{1}{\epsilon}W(\varphi^{0})
    + \int_{\partial\Omega}\gamma(\varphi^{0})
    +\sum_{k=1}^{m} \int_\Omega \rho^{k} g\cdot v^k.
  \end{aligned}
\end{equation}
\end{theorem}
\begin{proof}
The proof follows standard arguments and we refer to, e.g.,
\cite{AlandChen__MovingContactLine,Bonart2019b,GruenMetzger__CHNS_decoupled}. 
%
\end{proof}

\begin{lemma}
\label{lem:S:SolBound}
Let 
$\varphi^{m-1} \in H^1(\Omega) \cap L^\infty(\Omega)$
 and $v^{m-1} \in H_{\sigma}(\Omega)$ be given.
 Moreover, let $u \in U_{ad}$ be satisfied.
For every $m$, there exists exactly one  solution to \eqref{eq:S:1_NS}--\eqref{eq:S:4_CH2}.
Moreover, the following stability estimate holds:
\begin{align*}
 \|v^m\|_{H^1(\Omega)} &+ \|\varphi^m\|_{H^1(\Omega)} + \|\varphi^m\|_{C(\overline\Omega)} + \|\mu^m\|_{W^{1,3}(\Omega)}\\
 &\leq C\left( 
 \|\varphi^{m-1}\|_{H^1(\Omega)},
 \|\varphi^{m-1}\|_{L^\infty(\Omega)},
 \|v^{m-1}\|_{H^1(\Omega)} \right),
\end{align*}
where the right hand side is a polynomial of its arguments.
The result is not stable with respect to $\tau \to 0$ or $\epsilon \to 0$.
\end{lemma}
\begin{proof}
As systems \eqref{eq:S:1_NS} and \eqref{eq:S:3_CH1}--\eqref{eq:S:4_CH2} are decoupled,
we first argue the existence for the latter system.
The existence of a solution $\varphi^m\in H^1(\Omega)$, $\mu^m\in H^1(\Omega)$
to \eqref{eq:S:3_CH1}--\eqref{eq:S:4_CH2} can be shown by considering a suitable minimization problem,
compare, e.g., \cite{2011-HintHT_adaptiveFE_MoreauYosida_CH,2013-HintHK_adaptiveFE_MoreauYosida_CHNS}.
This solution satisfies
\begin{align}
  \label{eq:S:SolStab_H1}
  \|\varphi^m\|_{H^1(\Omega)} + \|\mu^m\|_{H^1(\Omega)} 
  \leq
  C \left( \|\varphi^{m-1}\|_{H^1(\Omega)}, \|\varphi^{m-1}\|_{L^\infty(\Omega)}, \|v^{m-1}\|_{H^1(\Omega)} \right),
\end{align}
where the right hand side is a polynomial of its arguments.
The uniqueness then follows from assuming the existence of two solutions, and exploiting the
monotonicity of $W_+^\prime$, see \cite[Thm. 4]{2016-GarHK_CHNS_AGG_linearStableTimeDisc}. 
 
Next we deduce the higher regularity for $\varphi^m$. Since $\varphi^m \in H^1(\Omega) \hookrightarrow L^6(\Omega)$
together with \eqref{ass:M:W_poly_bounded} it holds $W^\prime_+(\varphi^m) \in L^2(\Omega)$. Further, since $\gamma_u^\prime$
is Lipschitz continuous, we have from $\varphi^{m-1} \in H^{1/2}(\partial\Omega)$ that $\gamma_u^\prime(\varphi^{m-1}) \in H^{1/2}(\partial\Omega)$.
Now we employ \cite[Thm. 4.8]{TroelOptiSteu} to observe that $\varphi^m$ is continuous and that
\begin{align} 
  \label{eq:S:SolStab_varphi}
  \|\varphi^m\|_{H^1(\Omega)} + \|\varphi^m\|_{C(\overline \Omega)} 
  &\leq 
  C\left( \| \mu^m + W_-^\prime(\varphi^{m-1})\|_{L^2(\Omega)}
  + \|\gamma_u^\prime(\varphi^{m-1})\|_{H^{1/2}(\partial\Omega)}\right)\\  
  &\leq
  C \left( \|\mu^m\|_{L^2(\Omega)} + \|\varphi^{m-1}\|_{H^1(\Omega)}
  \right),
\end{align}
where in the last line we use  the properties of $W_-^\prime$
and the Lipschitz continuity of $\gamma_u^\prime$.

Now we deduce the higher regularity for $\mu^m$. Since $v^{m-1} \in H_\sigma(\Omega) \hookrightarrow L^6(\Omega)^n$,
we have $\divergence{(\varphi^{m-1}v^{m-1})} \in L^{3/2}(\Omega)$ and from
\cite[Thm. 1.9, Thm. 5.3]{2012_Dhamo_Diss_BoundaryControl}
we thus have $\mu^m \in W^{1,3}(\Omega)$. 
Additionally we obtain 
\begin{align}
    \label{eq:S:SolStab_mu}
  \|\mu^m\|_{W^{1,3}(\Omega)} 
\leq C\left( \|\varphi^m\|_{L^2(\Omega)},\|\varphi^{m-1}\|_{H^1(\Omega)},  \|\varphi^{m-1}\|_{L^\infty(\Omega)},\|v^{m-1}\|_{H^1(\Omega)} \right),
\end{align}
where the right hand side is a polynomial of its arguments.

Finally, the existence of a unique solution  $v^m \in H_\sigma(\Omega)$ to \eqref{eq:S:1_NS} readily follows by Lax--Milgram's theorem,
and leads to
\begin{align}
    \label{eq:S:SolStab_v}
  \|v^m\|_{H^1(\Omega)} 
  \leq 
  C\left(\|\varphi^{m-1}\|_{H^{1}(\Omega)},\|\varphi^{m-1}\|_{L^\infty(\Omega)}, \|\mu^{m}\|_{W^{1,3}(\Omega)},\|v^{m-1}\|_{H^1(\Omega)} \right),
\end{align}
where the right hand side is a polynomial of its arguments. 
Here we used that $\rho$ and $\eta$ are bounded independent of their arguments.

Summing over \eqref{eq:S:SolStab_varphi},\eqref{eq:S:SolStab_mu}, and \eqref{eq:S:SolStab_v} and using \eqref{eq:S:SolStab_H1} leads to the desired result.
\end{proof}

\begin{corollary}
By iterating the result from Lemma~\ref{lem:S:SolBound} we bound the solution
at any time instance by the initial data, namely
\begin{align}
\|v^m\|_{H^1(\Omega)} + \|\varphi^m\|_{H^1(\Omega)} + \|\varphi^m\|_{C(\overline\Omega)} + \|\mu^m\|_{W^{1,3}(\Omega)}
\leq C( \|\varphi^0\|_{H^1(\Omega)},\|\varphi^0\|_{L^\infty(\Omega)},\|v^0\|_{H^1(\Omega)}),
\end{align}
for $m=1,\ldots,M$,
or equivalently
\begin{align}
\|v_\tau\|_{l^\infty(H^1(\Omega))} 
+ \|\varphi_\tau\|_{l^\infty(H^1(\Omega))} 
+ \|\varphi_\tau\|_{l^\infty(C(\overline\Omega))} 
+ \|\mu_\tau\|_{l^\infty(W^{1,3}(\Omega))}
\leq C( \|\varphi^0\|_{H^1(\Omega)},\|\varphi^0\|_{L^\infty(\Omega)},\|v^0\|_{H^1(\Omega)}).
\end{align}
Here the constant depends polynomially on its arguments.
\end{corollary}

\begin{remark}
  We stress that the stability result from Lemma~\ref{lem:S:SolBound} is independent of the control $u$ and 
  the static contact angle
  because we already use the stated bounds $-1 \leq \cos(\theta_{eq}) + Bu \leq 1$ for $u \in U^0_{ad}$.
  Moreover, the unique solution from Lemma~\ref{lem:S:SolBound} can be found by Newton's method in function space,
  see Corollary~\ref{cor:O:NewtonsMethod}.
\end{remark}


\subsection{Results on higher regularity}

To prepare for the analysis of the optimal control problem in  Section~\ref{sec:O}  we next show results on  higher regularity for the solution to
\eqref{eq:S:1_NS}--\eqref{eq:S:4_CH2}.
This requires additional assumptions on the data.

\begin{assumption}
\label{ass:S:smoothData}
\hspace{1cm}
\begin{itemize}
  \item The boundary $\partial\Omega$ of $\Omega$ is of class $C^{1,1}$.
  \item The initial data $\varphi_0$ satisfies $\varphi_0 \in C^{0,1}(\overline \Omega)$, 
		i.e. $\varphi_0$ is Lipschitz continuous.
\end{itemize}
\end{assumption}

\begin{remark}
  In fact, the higher regularity  is only required to show continuity of the trilinear form $a(\cdot,\cdot,\cdot)$ under weak convergence, 
  see Lemma~\ref{lem:O:e_weak_conti}. 
  We show this by using sufficient regularity of $\nabla \mu^m$, namely $\mu^m \in H^{2}$.
  To achieve this regularity  we in turn require $C^{0,1}(\overline \Omega)$ regularity 
  for $\varphi^{m-1}$, appearing as a diffusion parameter in \eqref{eq:S:3_CH1}.
  Assumption~\ref{ass:S:smoothData} guarantees this regularity by guaranteeing $W^{2,p}(\Omega) \hookrightarrow C^{0,1}(\overline \Omega)$ 
  regularity  for $\varphi^m$.
\end{remark}

We start with showing Lipschitz continuity for $\varphi^m$ in Lemma~\ref{lem:S:phi_Lip} and show $H^2(\Omega)$-regularity for $\mu^m$ in Lemma~\ref{lem:S:mu_H2}.
\begin{lemma}
\label{lem:S:phi_Lip}
Let Assumption~\ref{ass:S:smoothData} hold and 
let $\varphi^{m-1} \in H^1(\Omega)  \cap C^{0,1}(\overline \Omega)$ and  $\mu^m\in W^{1,3}(\Omega)$ be given.
 
Then  $\varphi^m$, the solution to \eqref{eq:S:4_CH2}, satisfies $\varphi^{m} \in C^{0,1}(\overline \Omega)$
and
\begin{align*}
 \|\varphi^m\|_{C^{0,1}(\overline \Omega)} 
  \leq C \left(
   \|v^{m-1}\|_{H^1(\Omega)}, 
   \|\varphi^{m-1}\|_{H^{1}(\Omega)}, 
  \|\varphi^{m-1}\|_{C^{0,1}(\overline\Omega)},  
  \right),
\end{align*}
where the latter is a polynomial of its arguments.
\end{lemma}
\begin{proof}
From Lemma~\ref{lem:S:SolBound} we have $\varphi^m \in C(\overline \Omega) \hookrightarrow L^\infty(\Omega)$.
Thus we obtain $W_+^\prime(\varphi^m) \in L^{\infty}(\Omega)$.
Further $\mu^m \in  L^p(\Omega)$  for any $p< \infty$ by Sobolev embedding
and $W_-^\prime(\varphi^{m-1}) \in L^\infty(\Omega)$ since $\varphi^{m-1} \in C^{0,1}(\overline \Omega)$.
Finally, as $\varphi^{m-1} \in C^{0,1}(\overline \Omega)$ 
it holds $\varphi^{m-1}|_{\partial\Omega} \in C^{0,1}(\partial\Omega) \hookrightarrow W^{1,\infty}(\partial\Omega) \hookrightarrow W^{1-1/p,p}(\partial\Omega)$ 
for $p>n$. 
Combining these terms lead to 
$\varphi^m \in W^{2,p}(\Omega)$  by
\cite[Thm. 2.4.2.7]{Grisvard2011_EllipticProblems_nonsmoothDomains} with $p>n$.

The stability estimate now follows from \cite[Thm. 2.3.3.6]{Grisvard2011_EllipticProblems_nonsmoothDomains}
and the continuous embedding $W^{2,p}(\Omega) \hookrightarrow C^{0,1}(\overline \Omega)$, namely
\begin{align*}
 & \|\varphi^m\|_{C^{0,1}(\overline \Omega)}
  \leq C \|\varphi^m\|_{W^{2,p}(\Omega)}\\
 & \leq C \left( \| \mu^m + \frac{\sigma}{\epsilon}W^\prime_+(\varphi^m) + \frac{\sigma}{\epsilon}W^\prime_-(\varphi^{m-1})\|_{L^p(\Omega)}
  + \| -\left(\frac{r}{\tau} + \frac{S_\gamma}{2}\right)\varphi^{m-1} + \gamma^\prime(\varphi^{m-1}) \|_{W^{1-\frac{1}{p},p}(\partial\Omega)} \right)\\
  &\leq C\left( \|\mu^m\|_{W^{1,3}(\Omega)} + \|\varphi^m\|_{C(\overline\Omega)} + \|\varphi^{m-1}\|_{C(\overline\Omega)}%
  + \|\varphi^{m-1}\|_{C^{0,1}(\overline\Omega)} \right).
\end{align*}
In the last estimate we used the Lipschitz continuity of $\gamma_u^\prime$.
Using Lemma~\ref{lem:S:SolBound} leads to the final result.
\end{proof}

With the results from Lemma~\ref{lem:S:phi_Lip} we can now show higher regularity for $\mu^m$.
\begin{lemma}
\label{lem:S:mu_H2}
Let Assumption~\ref{ass:S:smoothData} hold and 
let $\varphi^{m-1} \in H^1(\Omega)  \cap C^{0,1}(\overline \Omega)$ and 
$v^{m-1}\in H_\sigma(\Omega)$ be given.

Then $\mu^m \in H^2(\Omega)$ and
\begin{align*}
  \|\mu^m\|_{H^2(\Omega)} 
  \leq C(\|v^{m-1}\|_{H^1(\Omega)},\|\varphi^{m-1}\|_{H^{1}(\Omega)},\|\varphi^{m-1}\|_{L^{\infty}(\Omega)}),
\end{align*}
where the latter denotes a polynomial of its arguments.
\end{lemma}
\begin{proof}

Since $\varphi^{m-1}$ is uniformly bounded, we can cut off $|\varphi^{m-1}|^2$ at some some large positive value $M$ und extend it linearly
without changing the  actually taken values.
We denote this modified function $|\varphi^{m-1}|^2_M$. 
Then for $\varphi^{m-1} \in C^{0,1}(\overline \Omega)$ 
we have that $|\varphi^{m-1}|^2_M \in C^{0,1}(\overline \Omega)$ is satisfied.

We have $\nabla\varphi^{m-1}v^{m-1} \in L^2(\Omega)$ since 
$\varphi^{m-1} \in H^1(\Omega) \hookrightarrow L^{4}(\Omega)$ 
and $v^{m-1}\in H_\sigma(\Omega) \hookrightarrow L^4(\Omega)^n$.
Now the regularity result follows from \cite[Thm. 2.4.2.7]{Grisvard2011_EllipticProblems_nonsmoothDomains},
while for the stability result 
\cite[Thm. 2.3.3.6]{Grisvard2011_EllipticProblems_nonsmoothDomains} is taken and provides
\begin{align*}
  \|\mu^m\|_{H^2(\Omega)} \leq C \| \tau^{-1} (\varphi^m-\varphi^{m-1}) + \nabla \varphi^{m-1}\cdot v^{m-1}\|_{L^2(\Omega)}.
\end{align*}
The final result again follows from applying Lemma~\ref{lem:S:SolBound}.
\end{proof}

We can summarize Lemma~\ref{lem:S:phi_Lip} and Lemma~\ref{lem:S:mu_H2} and state the following corollary.
\begin{corollary}
\label{cor:S:highRegul}
Let $\Omega$ and $\varphi_0$ satisfy Assumption~\ref{ass:S:smoothData}. 
Then  the solution to \eqref{eq:S:1_NS}--\eqref{eq:S:4_CH2} satisfies for $m=1,\ldots,M$
\begin{align*}
  v^m \in H_\sigma(\Omega),\quad \varphi^{m}\in H^1(\Omega) \cap C^{0,1}(\overline \Omega), \quad \mu^m \in H^2(\Omega).
\end{align*} 
It further holds
\begin{align*}
  \|v^m\|_{H^1(\Omega)} + \|\varphi^m\|_{H^1(\Omega)}+\|\varphi^m\|_{C^{0,1}(\overline\Omega)} + \|\mu^m\|_{H^2(\Omega)} 
  \leq C(\|\varphi_0\|_{H^1(\Omega)},\|\varphi_0\|_{C^{0,1}(\overline\Omega)}, \|v_0\|_{H^1(\Omega)}),
\end{align*}
or equivalently 
\begin{align*}
  \|v_\tau\|_{l^\infty(H^1(\Omega))}
  +\|\varphi_\tau\|_{l^\infty(H^1(\Omega))}
  +\|\varphi_\tau\|_{l^\infty(C^{0,1}(\overline\Omega))}
  +\|\mu_\tau\|_{l^\infty(H^2(\Omega))}
    \leq C(\|\varphi_0\|_{H^1(\Omega)},\|\varphi_0\|_{C^{0,1}(\overline\Omega)}, \|v_0\|_{H^1(\Omega)})
\end{align*}
where the constant depends polynomially on its arguments.
\end{corollary}

\begin{remark}
  Lemma~\ref{lem:S:phi_Lip} and Lemma~\ref{lem:S:mu_H2}
  stay true, if in Assumption~\ref{ass:S:smoothData} the regularity 
  \begin{center}
  $\Omega$ is of class $C^{1,1}$
  \end{center}
  is substituted by 
  \begin{center}
    $n=2$ and $\Omega$ is convex and polygonally bounded.
  \end{center}
  In the latter case the reference \cite[Thm. 2.4.2.7]{Grisvard2011_EllipticProblems_nonsmoothDomains}
  has to be changed to
  \cite[Thm. 4.3.2.4, Thm. 4.4.3.7]{Grisvard2011_EllipticProblems_nonsmoothDomains}. 
\end{remark}

\begin{notation}

Lemma~\ref{lem:S:SolBound} guarantees the existence of a solution operator for \eqref{eq:S:1_NS}--\eqref{eq:S:4_CH2}.
For a further investigation, we introduce the spaces
\begin{align*}
  Y &= \left( H_\sigma(\Omega) \times \left(H^1(\Omega) \cap L^\infty(\Omega)\right)  \times W^{1,3}(\Omega) \right)^M, 
  \quad
  y = (v_\tau, \varphi_\tau, \mu_\tau),\\ 
  Z &=  \left( (H_\sigma(\Omega) \times H^1(\Omega)  \times H^1(\Omega))^\star \right)^M.
\end{align*}
The norm of $y\in Y$ is defined by 
\begin{align*}
\|y\|_Y = \|(v_\tau,\varphi_\tau,\mu_\tau)\|_Y 
= 
\|v_\tau\|_{l^\infty(H^1(\Omega))}
+\|\varphi_\tau\|_{l^\infty(H^1(\Omega))}
+\|\varphi_\tau\|_{l^\infty(L^\infty(\Omega))}
+\|\mu_\tau\|_{l^\infty(W^{1,3}(\Omega))}
\end{align*}

With the above  spaces we introduce an operator $e: Y \times U \to Z$
to abbreviate \eqref{eq:S:1_NS}--\eqref{eq:S:4_CH2} for $m=1,\ldots,M$ as
\begin{align*}
  \left<z,e(y,u)\right>_{Z^\star,Z} = \sum_{m=1}^M
  &\left[
  \left( \frac{\rho^m+\rho^{m-1}}{2}v^m - \rho^{m-1}v^{m-1}, w^m \right) 
  + \tau a(\rho^{m-1}v^{m-1} + J^m,v^m,w^m) \right.\\
  &\left.\vphantom{\frac{1}{2}}
  +\tau (2\eta^{m}Dv^m,Dw^m)
  + \tau (\varphi^{m-1}\nabla \mu^m,w^m) 
  - \tau (g\rho^{m},w^m)
  \right]\\
  %
  &+\left[\vphantom{\frac{\tau^2}{2}}
  (\varphi^m - \varphi^{m-1},\Psi^m)
  -\tau (\varphi^{m-1}v^{m-1},\nabla \Psi^m) \right.\\
  &\left. \vphantom{\frac{\tau^2}{2}}
  +\frac{\tau^2}{\rho_{\min}}(|\varphi^{m-1}|^2\nabla \mu^m,\nabla \Psi^m)
  +\tau b(\nabla \mu^m,\nabla \Psi^m)
  \right]\\
  %
  &+\left[ \vphantom{\frac{S_\gamma}{2}}
  \tau\sigma\epsilon(\nabla \varphi^m,\nabla \Phi^m)  
  +\tau \frac{\sigma}{\epsilon} (W_+^\prime(\varphi^m) + W_-^\prime(\varphi^{m-1}),\Phi^m)
  -\tau (\mu^m,\Phi^m) \right.\\
  &\left.
  +r(\varphi^m-\varphi^{m-1},\Phi^m)_{\partial\Omega}
  +\tau \left( \frac{S_\gamma}{2}(\varphi^m-\varphi^{m-1}) + \gamma_u^\prime(\varphi^{m-1}),\Phi^m \right)_{\partial\Omega}
  \right]
\end{align*}
with $z = ( (w^m)_{m=1}^M, (\Psi^m)_{m=1}^M), (\Phi^m)_{m=1}^M )$.
With this notation 
\eqref{eq:S:1_NS}--\eqref{eq:S:4_CH2} for $m=1,\ldots,M$
reduces to 
\begin{align}
  \label{eq:S:eyu}
  e(y,u) = 0 \in Z
\end{align}
and Lemma~\ref{lem:S:SolBound} ensures, that for any $u\in U^0_{ad}$ there exists a unique $y\in Y$ such that $e(y,u) = 0$
and $\|y\|_Y \leq C$, where $C>0$ is independent of $u\in U_{ad}^0$.
\end{notation}


\subsection{Continuity under weak convergence}
\label{ssec:S:conti}
Here we show that the solution operator for \eqref{eq:S:1_NS}--\eqref{eq:S:4_CH2} is continuous under weak convergence
and start with a preparatory lemma.

\begin{lemma}
\label{lem:O:LinfLq_strongConv}
Let $(g_i)_{i\in\N}$ denote a sequence that converges strongly to $g$ in $L^p(\Omega)$ for some $p\geq 1$.
Further there exists $C>0$ such that  $|g_i| \leq C$ independent of $i$.
Let $\Psi \in L^q(\Omega)$ denote a fixed function with $q< \infty$.

Then $\|g_i\Psi - g\Psi\|_{L^q(\Omega)} \to 0$, i.e. $(g_i\Psi) \to g\Psi$ strongly in $L^q(\Omega)$.
\end{lemma}
\begin{proof}
This is a direct result from Lebesgue's general convergence theorem, see e.g.~\cite[Thm. 3.25]{Alt_lineareFunktionalAnalysis}.
Since $(g_i)_{i\in\N}$ converges strongly in some $L^p(\Omega)$ a subsequence converges pointwise almost everywhere, 
and the same holds for $( g_i \Psi)_{i\in\N}$.
Further $|g_i \Psi|^q \leq C |\Psi|^q \in L^1(\Omega)$.

From \cite[Thm. 3.25]{Alt_lineareFunktionalAnalysis} in this situation we obtain $g_i \Psi \in L^q(\Omega)$ and
$(g_i \Psi) \to g\Psi$ strongly in $L^q(\Omega)$.

\end{proof}

\begin{lemma}
\label{lem:O:e_weak_conti}
Let Assumption~\ref{ass:S:smoothData} hold.
Let $(u_i)_{i\in \N}\subset U^0_{ad}$ denote a weakly converging sequence to some $u\in U^0_{ad}$ 
and let $(y_i)_{i\in\N} \in Y$ denote  a weakly converging subsequence to some $y\in Y$
with $e(y_i,u_i) = 0$ for all $i\in \N$.
Then $e(y,u) = 0$, i.e. $e$ is continuous under weak convergence.
\end{lemma}
\begin{proof}
We discuss the operator $e$ for a fixed $m \in \{1,\ldots,M\}$ and denote subsequences 
by the same index $i$. All results hold for certain subsequences.
We only discuss the convergence of the nonlinear terms. The linear terms directly converge by definition of weak convergence.
Unless explicitly mentioned, all results are valid without the higher regularity that Assumption~\ref{ass:S:smoothData}
guarantees.

\medskip

\noindent\textbf{Convergence in equation~\eqref{eq:S:3_CH1}:}
\begin{itemize}
  \item It holds $(\varphi_i^{m-1}v_i^{m-1},\nabla \Psi) \to (\varphi^{m-1}v^{m-1},\nabla \Psi)$ 
by the compact embedding $H^1(\Omega) \hookrightarrow L^4(\Omega)$
that we use for $\varphi_i^{m-1}$ and $v_i^{m-1}$.
\item Further, we have $\varphi_i^{m-1} \to \varphi^{m-1}$ strongly in $L^2(\Omega)$ by compact embedding $H^1(\Omega) \hookrightarrow L^2(\Omega)$
and thus pointwise almost everywhere for a subsequence.
Further $\varphi_i^{m-1}$ is  uniformly bounded in $L^\infty(\Omega)$.
From Lemma~\ref{lem:O:LinfLq_strongConv}
we observe $|\varphi_i^{m-1}|^2 \nabla \Psi \to |\varphi^{m-1}|^2 \nabla \Psi$ strongly in $L^2(\Omega)$ and thus
$(|\varphi_i^{m-1}|^2\nabla \mu_i^m,\nabla \Psi) \to (|\varphi^{m-1}|^2\nabla \mu^m,\nabla \Psi)$.
\end{itemize}

\noindent\textbf{Convergence in equation~\eqref{eq:S:4_CH2}:}
\begin{itemize}
  \item 
The term $(W^\prime_+(\varphi^m_i) + W^\prime_-(\varphi^{m-1}_i),\Phi)$ converges to
$(W^\prime_+(\varphi^m) + W^\prime_-(\varphi^{m-1}),\Phi)$ by Lebesgues's general convergence theorem,
using  the bounds $|W^\prime_+(\varphi^m)| \leq C(1+|\varphi^m|^3)$
and $|W^\prime_-(\varphi^{m-1})| \leq C(1+|\varphi^{m-1}|^3)$ and the compact embedding $H^1(\Omega) \hookrightarrow L^5(\Omega)$ to
obtain a strongly converging subsequence.
\item Further the trace operator $Tr:H^1(\Omega) \to L^2(\partial\Omega)$ is compact, thus $\varphi_i^m|_{\partial\Omega}$
is strongly converging in $L^2(\partial\Omega)$, see \cite[Thm. 6.2]{Necas_DirectMethodsEllipticEquations}.
\item Using the structure of
$\gamma_u^\prime(\varphi_i^{m-1}) = \sigma_{lg} (\cos(\theta_{eq})+B_m u_i) \vartheta^\prime(\varphi_i^{m-1})$
we next show that $\vartheta^\prime(\varphi_i^{m-1})\Phi$ is
strongly converging with respect to $L^2(\partial\Omega)$ and thus 
$\gamma_u^\prime(\varphi^{m-1}_i)\Phi \to \gamma_u^\prime(\varphi^{m-1})\Phi$.
The sequence $\vartheta^\prime(\varphi^{m-1}_i)$ converges pointwise almost everywhere and is uniformly bounded in $L^\infty(\partial\Omega)$
by construction.
Thus by Lemma~\ref{lem:O:LinfLq_strongConv}
$\vartheta^\prime(\varphi^{m-1}_i)\Phi \to \vartheta^\prime(\varphi^{m-1})\Phi$ strongly in $L^2(\partial\Omega)$
which yields
$ (B_mu_i\vartheta^\prime(\varphi_i^{m-1}),\Phi)_{\partial\Omega} \to (B_mu\vartheta^\prime(\varphi^{m-1}),\Phi)_{\partial\Omega}$,
where we use that $B$ is linear and thus continuous under weak convergence.
\end{itemize}

\noindent\textbf{Convergence in equation~\eqref{eq:S:1_NS}:}
Besides the trilinear form $a(\rho_i^{m-1}v_i^{m-1} + J^m_i,v_i^m,w)$ the  convergence follows straightforwardly with
the arguments mentioned so far.
\begin{itemize}
  \item We first consider $( (\rho_i^{m-1}v_i^{m-1}\nabla) v_i^{m},w)$. By 
  Lemma~\eqref{lem:O:LinfLq_strongConv} we have
$\rho_i^{m-1} w \to \rho^{m-1}w$ strongly in $(L^4(\Omega))^n$. Together with the compact embedding $(H^1(\Omega))^n \hookrightarrow (L^4(\Omega))^n$
we obtain $( (\rho_i^{m-1}v_i^{m-1}\nabla) v_i^{m} , w) \to ( (\rho^{m-1}v^{m-1}\nabla) v^{m} , w)$.
\item Similarly we obtain
$( (\rho_i^{m-1}v_i^{m-1}\nabla)  w , v_i^{m} ) \to ( (\rho^{m-1}v^{m-1}\nabla)  w , v^{m} )$ 
using strong convergence
$\rho_i^{m-1}\nabla w \to \rho^{m-1}\nabla w$ in $(L^2(\Omega))^n$.
\end{itemize}
 Finally we consider the terms including $J^m = -b\rho^\prime(\varphi^{m})\nabla \mu^m$. 
Since $(y_i)$ solves \eqref{eq:S:1_NS}--\eqref{eq:S:4_CH2} and due to Assumption~\ref{ass:S:smoothData} 
Corollary~\ref{cor:S:highRegul} applies. 
We stress that this is the only term, that requires the results on higher regularity.
\begin{itemize}
\item We start with 
$( (\rho^\prime(\varphi_i^m) \nabla \mu_i^m\nabla)  w \cdot v_i^{m} ) \to ( (\rho^\prime(\varphi_i^m) \nabla \mu_i^m\nabla)  w \cdot v^{m} )$.
Here we use Lemma~\ref{lem:O:LinfLq_strongConv} to obtain 
$\rho^\prime(\varphi_i^m) \nabla w \to \rho^\prime(\varphi^m) \nabla w$ strongly in $(L^2(\Omega))^n$.
By compact embedding we have $v_i^m \to v^m$ strongly in $(L^{4}(\Omega))^n$, 
and further from the compact embedding $H^2(\Omega) \hookrightarrow W^{1,4}(\Omega)$ we have strong convergence
$\nabla \mu_i^m \to \nabla \mu^m$ in $(L^{4}(\Omega))^n$.
This yields the desired convergence.
\item For the convergence
$( (\rho^\prime(\varphi_i^m) \nabla \mu_i^m\nabla)   v_i^{m} \cdot w ) \to ( (\rho^\prime(\varphi_i^m) \nabla \mu_i^m\nabla)   v_i^m \cdot w )$
we first note, that $\rho^{\prime}(\varphi_i^m) w \to \rho^\prime(\varphi^m)w$ strongly in $(L^6(\Omega))^n$ 
by the continuous embedding $(H^1(\Omega))^n \hookrightarrow (L^6(\Omega))^n$ and Lemma~\ref{lem:O:LinfLq_strongConv}.
Further we have $\nabla \mu_i^m \to \nabla \mu^m$ strongly in $(L^3(\Omega))^n$ by the compact embedding
 $H^2(\Omega) \hookrightarrow W^{1,3}(\Omega)$. Together with the convergence $v_i^m \to v^m$ weakly in $(L^2(\Omega))^n$ the result follows.
\end{itemize}
\end{proof}


\subsection{Fr{\'e}chet differentiability of the forward model}
\label{ssec:S:diff}

By a straightforward  calculation we observe that for
 given $u \in U_{ad}^0$, the operator $e(\cdot,u)$ is Gateaux differentiable with respect to the first component.
The Gateaux derivative at $y = ( (v^m)_{m=1}^M, (\varphi^m)_{m=1}^M, (\mu^m)_{m=1}^M ) \in Y$ in direction  $d_y = ( (d_v^m)_{m=1}^M, (d_\varphi^m)_{m=1}^M, (d_\mu^m)_{m=1}^M )$ is given by
\begin{equation}
\label{eq:S:edy}
\begin{aligned}
\left< e_y(y,u)d_y,z\right>_{Z^\star,Z} = 
\sum_{m=1}^M
\left[
\frac{1}{2}(\rho^\prime(\varphi^m)d_\varphi^m,v^mw^m + v^{m+1}w^{m+1})
+\frac{1}{2}( (\rho^m + \rho^{m-1})d_v^m,w^m)\right.\\
\left.
-( \rho^\prime(\varphi^m)d_\varphi^m v^m + \rho^m d_v^m,w^{m+1})
\right.
\\
+\left. 
\tau a(\rho^\prime(\varphi^m)d_\varphi^m v^{m} + \rho^m d_v^m,
v^{m+1},w^{m+1})\right.\\
\left.
+\tau a(-b\rho^{\prime\prime}(\varphi^m)d_\varphi^m\nabla \mu^{m}-b\rho^{\prime}(\varphi^{m})\nabla d_\mu^m,v^m,w^m) 
\right.\\
\left.
+\tau a(\rho^{m-1}v^{m-1}+J^m,d_v^m,w^m)
\right.\\
+\tau(2\eta^\prime(\varphi^m)d_\varphi^m Dv^{m},Dw^{m})
+\tau(2\eta^{m}Dd_v^m,Dw^m)\\
+\tau(d_\varphi^m\nabla \mu^{m+1},w^{m+1}) 
+\tau(\varphi^{m-1}\nabla d_\mu^m,w^m)\\
\left.\vphantom{\frac{1}{2}}
-\tau(g\rho^\prime(\varphi^m)d_\varphi^m,w^m)
\right]\\
%
%
+
\left[ \vphantom{\frac{\tau^2}{\rho_{\min}}}
(d_\varphi^m,\Psi^m-\Psi^{m+1})
 -\tau (d_\varphi^m v^m+\varphi^{m} d_v^m,\nabla\Psi^{m+1})
 \right.\\
 \left.
 +\frac{\tau^2}{\rho_{\min}}( 2d_\varphi^m\nabla \mu^{m+1},\nabla \Psi^{m+1})
+\frac{\tau^2}{\rho_{\min}}( |\varphi^{m-1}|^2\nabla d_\mu^{m},\nabla \Psi^{m})
+\tau b(\nabla d_\mu^m,\nabla \Psi^m)
\right]\\
%
%
+\left[
\tau\sigma\epsilon(\nabla d_\varphi^m,\nabla \Phi^m) 
+ \frac{\tau\sigma}{\epsilon}(W^{\prime\prime}_+(\varphi^m)d_\varphi^m,\Phi^m) 
+ \frac{\tau\sigma}{\epsilon}(W^{\prime\prime}_-(\varphi^m)d_\varphi^m,\Phi^{m+1}) 
-\tau(d_\mu^m,\Phi^m)\right.\\
\left.
+ r(d_\varphi^m,\Phi^m-\Phi^{m+1})_{\partial\Omega}
+\tau\left(
\frac{S_\gamma}{2}(d_\varphi^m,\Phi^m-\Phi^{m+1})_{\partial\Omega}
+(\gamma_u^{\prime\prime}(\varphi^m)d_\varphi^m,\Phi^{m+1})_{\partial\Omega}
\right)
\right].
\end{aligned}
\end{equation} 
The square brackets $[\cdot]$ indicate, that the included terms stem from differentiating the same equation. 
Moreover, we use the convention, that functions with index $M+1$ are defined as zero.

 \begin{theorem}
 \label{thm:S:e_y_Frechet}

 Let $u\in U^0_{ad}$ be given. 
 The operator $e(\cdot,u):Y\to Z$ is  Fr{\'e}chet differentiable.  
 \end{theorem}
 \begin{proof}
 As \eqref{eq:S:eyu} describes a time stepping scheme, it is sufficient to consider only one time instance
 as the derivative has a lower block triangular structure.
 Moreover, the systems on every time instance are sequentially coupled that leads to a triangular 
 structure also of the diagonal blocks.
 Therefor we start with discussing \eqref{eq:S:3_CH1}--\eqref{eq:S:4_CH2} and
 discuss \eqref{eq:S:1_NS} afterwards.
 
 The Fr{\'e}chet-differentiability of \eqref{eq:S:3_CH1}--\eqref{eq:S:4_CH2} without the 
 additional terms from transport and transport decoupling directly follows as e.g. in \cite{2011-HintHT_adaptiveFE_MoreauYosida_CH}.
 The additional terms from the derivative with respect to $\varphi^{m-1}$ are in a lower diagonal block. 
 Note that here it is required, that we use variations in $L^\infty(\Omega)$.
 
 Since \eqref{eq:S:1_NS} is a linear system it is Fr{\'e}chet differentiable.
 \end{proof}

\begin{theorem}
Let $u\in U_{ad}^0$ and $y \in Y$ be given. Then there exists a unique solution $d_y$
to the linear system of equations given by \eqref{eq:S:edy}. 
This solution satisfies 
\begin{align*}
\|d_y\|_Y \leq 
C&\left(
\|\varphi^{0}\|_{H^1(\Omega)}, 
\|\varphi^{0}\|_{L^\infty(\Omega)},
\|v^{0}\|_{H^1(\Omega)}\right)
\end{align*}
Thus the linear operator $e_y(\cdot,u)$ is uniformly invertible.
\end{theorem} 

\begin{proof}
Equation~\eqref{eq:S:edy} is a time stepping scheme and on every time step we have
to solve the following linear system of equations
\begin{align}
\left(\frac{\rho^\prime(\varphi^m)d_\varphi^m + \rho^\prime(\varphi^{m-1})d_\varphi^{m-1}}{2} v^m + \frac{\rho^m+\rho^{m-1}}{2}d_v^m,w^m\right)& \nonumber\\
-(\rho^\prime(\varphi^{m-1})d_{\varphi}^{m-1}v^{m-1}+\rho^{m-1}d_v^{m-1},w^m) &\nonumber\\
+\tau a(\rho^\prime(\varphi^{m-1})d_\varphi^{m-1} v^{m-1} + \rho^{m-1} d_v^{m-1},v^{m},w^{m})&\nonumber\\
+\tau a(-b\rho^{\prime}(\varphi^{m})\nabla d_\mu^m-b\rho^{\prime\prime}(\varphi^{m})d_\varphi^{m}\nabla \mu^{m},v^m,w^m) 
+\tau a(\rho^{m-1}v^{m-1}+J^m,d_v^m,w^m)&\nonumber\\
+\tau(2\eta^{m}Dd_v^m,Dw^m) + \tau(2\eta^\prime(\varphi^{m})d_\varphi^{m} Dv^{m},Dw^{m})&\nonumber\\
+\tau(\varphi^{m-1}\nabla d_\mu^m,w^m)
-\tau(g\rho^\prime(\varphi^m)d_\varphi^m,w^m)
+\tau(d_\varphi^{m-1}\nabla \mu^{m},w^{m})&= 0, \label{eq:S:lin1}\\
%
%
%
(d_\varphi^m-d_\varphi^{m-1},\Psi^m)
-\tau(d_\varphi^{m-1}v^{m-1} + \varphi^{m-1}d_v^{m-1},\nabla\Psi^m) &\nonumber\\
+\frac{\tau^2}{\rho_{\min}}(2d_\varphi^{m-1}\nabla\mu^m,\nabla\Psi^m)
+\frac{\tau^2}{\rho_{\min}}(|\varphi^{m-1}|^2\nabla d_\mu^m,\nabla\Psi^m) 
+\tau b(\nabla d_\mu^m,\nabla \Psi^m)&=0, \label{eq:S:lin2}\\
%
%
%
\tau\sigma\epsilon(\nabla d_\varphi^m,\nabla\Phi^m)
+\frac{\tau\sigma}{\epsilon}(W_+^{\prime\prime}(\varphi^m)d_\varphi^m,\Phi^m) 
+\frac{\tau\sigma}{\epsilon}(W_-^{\prime\prime}(\varphi^{m-1})d_\varphi^{m-1},\Phi^m) 
-\tau(d_\mu^m,\Phi^m)&\nonumber\\
+r(d_\varphi^m-d_\varphi^{m-1},\Phi^m)_{\partial\Omega}
+\tau
\left(
\frac{S_\gamma}{2}(d_\varphi^m-d_\varphi^{m-1},\Phi^m)_{\partial\Omega}
+ (\gamma^{\prime\prime}_u(\varphi^{m-1})d_\varphi^{m-1},\Phi^m)_{\partial\Omega}
\right)&=0. \label{eq:S:lin3}
\end{align}

It is sufficient to show the results for one time step only.
Further we note, that \eqref{eq:S:lin1} 
is decoupled from \eqref{eq:S:lin2}--\eqref{eq:S:lin3} and we first discuss  
\eqref{eq:S:lin2}--\eqref{eq:S:lin3}.

The existence for \eqref{eq:S:lin2}--\eqref{eq:S:lin3} follows from a Galerkin approach.
Since \eqref{eq:S:lin2}--\eqref{eq:S:lin3} is a linear system in $d_\varphi^m,d_\mu^m$,
the existence of a  finite dimensional approximation follows from its uniqueness that
can be shown by considering the difference between two solutions and using these differences
at test functions in \eqref{eq:S:lin2}--\eqref{eq:S:lin3}.

We next show the a-priori bound for this solution, the existence and uniqueness for 
\eqref{eq:S:lin2}--\eqref{eq:S:lin3} then follows immediately by using a Galerkin approach.
For given $\varphi^{m-1} \in L^\infty(\Omega)$, $\varphi^m\in L^\infty(\Omega)$, 
$\mu^m \in W^{1,3}(\Omega)$, $v^{m-1}\in H^1(\Omega)$, 
$d_\varphi^{m-1} \in L^\infty(\Omega)\cap H^1{\Omega}$, $d_v^{m-1} \in H^1(\Omega)$ 
we use $\Psi^m \equiv d_\mu^m$ in \eqref{eq:S:lin2}
and $\Phi^m \equiv \frac{1}{\tau}(d_\varphi^m-d_\varphi^{m-1})$ in \eqref{eq:S:lin3} and add these equations to obtain
\begin{align*}
\tau b\|\nabla d_\mu^m\|^2 + \frac{\tau^2}{\rho_{\min}}(|\varphi^{m-1}|^2\nabla d_\mu^m,d_\mu^m)
+\frac{\tau^2}{\rho_{\min}}(2d_\varphi^{m-1}\nabla \mu^m,\nabla d_\mu^m)
-\tau(d_\varphi^{m-1}v^{m-1} + \varphi^{m-1}d_v^{m-1},\nabla d_\mu^m)&\\
\sigma\epsilon(\nabla d_\varphi^m,\nabla (d_\varphi^m-d_\varphi^{m-1}))
+\frac{\sigma}{\epsilon}(W_+^{\prime\prime}(\varphi^m)d_\varphi^m,d_\varphi^m-d_\varphi^{m-1})
+\frac{\sigma}{\epsilon}(W_-^{\prime\prime}(\varphi^{m-1})d_\varphi^{m-1},d_\varphi^m-d_\varphi^{m-1})&\\
+ \frac{r}{\tau} \|d_\varphi^m-d_\varphi^{m-1}\|^2_{\partial\Omega}
+ \frac{S_\gamma}{2}\|d_\varphi^m-d_\varphi^{m-1}\|^2_{\partial\Omega}
+ (\gamma^{\prime\prime}_u(\varphi^{m-1})d_\varphi^{m-1},d_\varphi^m-d_\varphi^{m-1})_{\partial\Omega}
&=0.
\end{align*}
Using $2a(a-b) = a^2 + (a-b)^2 - b^2$, and $|\varphi^{m-1}|^2 \geq 0$ we proceed
\begin{align*}
\tau b\|\nabla d_\mu^m\|^2
+ \frac{\sigma\epsilon}{2}\|\nabla d_\varphi^m\|^2 
+ \frac{\sigma\epsilon}{2}\|\nabla(d_\varphi^m-d_\varphi^{m-1})\|^2
-\frac{\sigma\epsilon}{2}\|\nabla d_\varphi^{m-1}\|^2&\\
+\frac{\sigma}{\epsilon}(W_+^{\prime\prime}(\varphi^m),
(d_\varphi^m-d_\varphi^{m-1})^2)
+ \left(\frac{r}{\tau}+\frac{S_\gamma}{2}\right) \|d_\varphi^m-d_\varphi^{m-1}\|^2_{\partial\Omega}
&\\
\leq
- \frac{\tau^2}{\rho_{\min}}(2d_\varphi^{m-1}\nabla \mu^m,\nabla d_\mu^m)
+ \tau(d_\varphi^{m-1}v^{m-1} + \varphi^{m-1}d_v^{m-1},\nabla d_\mu^m)&\\
-\frac{\sigma}{\epsilon}(W_+^{\prime\prime}(\varphi^m)d_\varphi^{m-1},d_\varphi^m-d_\varphi^{m-1})
-\frac{\sigma}{\epsilon}(W_-^{\prime\prime}(\varphi^{m-1})d_\varphi^{m-1},d_\varphi^m-d_\varphi^{m-1})
- (\gamma^{\prime\prime}_u(\varphi^{m-1})d_\varphi^{m-1},d_\varphi^m-d_\varphi^{m-1})_{\partial\Omega}.
\end{align*}
From testing \eqref{eq:S:lin2} with $\Psi \equiv 1$ we observe
$(d_\varphi^m-d_\varphi^{m-1},1) = 0$ and thus we enjoy mass conservation for $d_\varphi^m$, $m=1,\ldots,M$
and especially we can use the inequality of Poincar\'e--Friedrichs to estimate 
$\|d_\varphi^m-d_\varphi^{m-1}\| \leq C\|\nabla(d_\varphi^m-d_\varphi^{m-1})\|$.

Using the inequalities of H\"older and Young we can thus obtain, that
it holds
\begin{align}
\label{eq:S:lin_bnd_grad}
\|\nabla d_\mu^m\| + \|\nabla d_\varphi^m\| 
\leq C(\|\varphi^m\|_{L^\infty(\Omega)},\|\varphi^{m-1}\|_{L^\infty(\Omega)},\|\nabla \mu^m\|_{L^2(\Omega)},
\|v^{m-1}\|_{L^2(\Omega)},\|d_\varphi^{m-1}\|_{L^\infty(\Omega)},\|d_v^{m-1}\|_{L^2(\Omega)}),
\end{align}
where $C$ is a polynomial of its arguments. Note that we only have $W_+^{\prime\prime}(\varphi^m)\geq 0$.

Using the mass conservation of $d_\varphi^m$ we in fact can bound $\|d_\varphi^m\|_{H^1(\Omega)}$ by the right hand side.
Finally we use $\Phi^m \equiv 1$ in \eqref{eq:S:lin3} 
together with this bound for $\|d_\varphi^m\|_{H^1(\Omega)}$ to obtain, that $(d_\mu^m,1)$ is uniformly bounded
from which we conclude by Poincar\'e--Friedrichs inequality that $\|d_\mu^m\|_{H^1(\Omega)}$ is also bounded and we obtain
\begin{align}
\|d_\mu^m\|_{H^1(\Omega)}& + \|d_\varphi^m\| _{H^1(\Omega)}\\
\leq C&\left(
\|\varphi^m\|_{L^\infty(\Omega)},
\|\varphi^{m-1}\|_{L^\infty(\Omega)},
\|\nabla \mu^m\|_{L^2(\Omega)},
\|v^{m-1}\|_{L^2(\Omega)},\right.\\
&\left.\|d_\varphi^{m-1}\|_{L^\infty(\Omega)},
\|d_v^{m-1}\|_{L^2(\Omega)}\right),
\end{align}
Again by using \cite[Thm. 4.8]{TroelOptiSteu} we obtain that additionally $\|d_\varphi^m\|_{C(\overline \Omega)}$
is bounded and the boundedness of $\|d_\mu^m\|_{W^{1,3}(\Omega)}$ follows 
from \cite[Thm. 1.9, Thm. 5.3]{2012_Dhamo_Diss_BoundaryControl}.
\begin{align}
\|d_\varphi^m\|_{H^1(\Omega)}&
+\|d_\varphi^m\|_{C(\overline \Omega)}
+\|d_\mu^m\|_{W^{1,3}(\Omega)}\\ 
\leq C&\left( 
\|\varphi^{m-1}\|_{L^\infty(\Omega)},
\|\varphi^{m}\|_{L^\infty(\Omega)},
\|\nabla \mu^{m}\|_{L^3(\Omega)},
\|v^{m-1}\|_{L^2(\Omega)},
\|v^{m}\|_{H^1(\Omega)},\right.\\
&\left.\|d_\varphi^{m-1}\|_{L^\infty(\Omega)},
\|\nabla d_\mu^{m-1}\|_{L^3(\Omega)},
\|d_v^{m-1}\|_{L^2(\Omega)}
\right).
\end{align}

Having the unique solution to \eqref{eq:S:lin2}--\eqref{eq:S:lin3} at hand,
in \eqref{eq:S:lin1} we solve the same system for $d_v^m$ as we solve in \eqref{eq:S:1_NS} for $v^m$.
The existence of a unique solution again follows immediately from Lax--Milgram's theorem, together with the bound
\begin{align}
\|d_v^m\|_{H^1(\Omega)} 
\leq& C(\|v^{m-1}\|_{L^2(\Omega)},\|v^m\|_{H^1(\Omega)},\|\varphi^{m-1}\|_{L^\infty(\Omega)},
\|\nabla \mu^m\|_{L^3(\Omega)},\\
&\|d_\varphi^m\|_{L^\infty(\Omega)},\|d_\varphi^{m-1}\|_{L^\infty(\Omega)},
\|d_v^{m-1}\|_{L^2(\Omega)},\|\nabla d_\mu^m\|_{L^3(\Omega)}).
\end{align}
Iterating these estimates yields the desired result.
\end{proof}

\begin{corollary}
\label{cor:O:NewtonsMethod}
Given $u\in U_{ad}^0$,
Newton's method in function space can be used to find
 the unique solution $y \in Y$ for $e(y,u) = 0$ that is guaranteed by Lemma~\ref{lem:S:SolBound} 
and yields local superlinear convergence.
\end{corollary}

\section{The optimization problem}
\label{sec:O}
In this section we introduce and discuss the optimization problem under investigation.
We assume, that we can influence the contact angle at the contact line by using a suitable control mechanism.
For this we introduce a general normed and reflexive control space $U$
and an injective,  linear and continuous control operator $B:U \to L^2(0,T;L^2(\partial\Omega))$, 
that maps given controls $u$ to suitable control actions,
see Remark~\ref{rm:M:ContactEnergy}.
The space $U$ and the operator $B$ encapsulate the model for the actual control action. 

We consider the following optimal control problem
\begin{equation}
  \label{prob:O:P}  \tag{P}
 \begin{aligned}
  \min_{u\in U, y \in Y} J(y,u)& = \frac{1}{2}\int_0^T\int_\Omega |\varphi_\tau-\varphi_d|^2\dx\dt + \frac{\alpha}{2}\|Bu\|^2_{L^2(0,T;L^2(\partial\Omega))}\\
  \mbox{subject to } & e(y,u) = 0,\\
                     & \cos (\theta_{\min}) \leq Bu  + \cos (\theta_{eq})   \leq \cos (\theta_{\max}).
 \end{aligned}
\end{equation} 
for a constant $\alpha >0$, where  $\varphi_d$ denotes a given distribution of the phases, that we want to obtain.
Here $0 < \theta_{\min} < \theta_{\max} < \pi$  are given minimum and maximum static 
contact angles that can be achieved and we define 
\begin{align*}
U_{ad}:= \{u\in U\,|\,\cos (\theta_{\min}) \leq Bu  + \cos (\theta_{eq})  \leq \cos (\theta_{\max}) \} 
\subset U_{ad}^0.
\end{align*}
Since $B$ is injective, $U_{ad}$ is a bounded set.
Here we use $\|Bu\|^2_{L^2(0,T;L^2(\partial\Omega))}$ as regularization in \eqref{prob:O:P}, 
because its understanding as deviation from the equilibrium angle makes it easier to interprete
than the common regularization $\|u\|^2_U$.
Finally we note, that $\varphi_\tau$ is piecewise constant with repect to time
and thus it holds
\begin{align*}
\frac{1}{2}\int_0^T\int_\Omega |\varphi_\tau-\varphi_d|^2\dx\dt
=\frac{1}{2}\sum_{m=1}^M \tau \int_\Omega |\varphi^m - \varphi_d^m|^2\dx,
\end{align*}
where $\varphi_d^m := \frac{1}{\tau}\int_{t_{m-1}}^{t_{m}}\varphi_d(t)\dt$.


We can now use the results from Section~\ref{ssec:S:conti} and Section~\ref{ssec:S:diff} to show the existence of 
at least one optimal control to \eqref{prob:O:P} and to state first order optimality conditions.

\begin{theorem}[Existence of an optimal solution]
\label{thm:O:exOpt}
There exists at least one optimal solution $(u^\star,y^\star)$ to \eqref{prob:O:P}.
\end{theorem}
\begin{proof}
This follows from the  direct method by considering an infimizing sequence and using Lemma~\ref{lem:O:e_weak_conti}. 
We refer to \cite[Ass. 1.44, Thm. 1.45]{HinzePinnauUlbrichUlbrich}.
\end{proof}


\begin{theorem}[First order optimality conditions]
\label{thm:O:OPT}
 Let $(\overline y, \overline u)$ denote a solution to \eqref{prob:O:P}.
 Then it satisfies the following first order optimality conditions with some adjoint variable $\overline p \in Z^\star$.
 \begin{align*}
 e(\overline y,\overline u) &= 0,\\
 e_y(\overline y,\overline u)^\star \overline p &= -(\overline \varphi_\tau-\varphi_d)\\
 \overline u \in  U_{ad},
 \quad 
 \left<\alpha B^\star Bu(\overline u) + e_u(\overline y,\overline u)^\star \overline p,u- \overline u\right>_{U^\star,U} &\geq 0 
 \quad 
 \forall u \in U_{ad}. 
 \end{align*}
 Here $e_y(\overline y, \overline u)^\star$ denotes the adjoint operator of $e_y(\overline y, \overline u)$
 and correspondingly for $e_u(\overline y,\overline u)^\star$.
\end{theorem}
\begin{proof}
See \cite[Cor. 1.3]{HinzePinnauUlbrichUlbrich}.
\end{proof}

For a practical implementation, we state the optimality system from Theorem~\ref{thm:O:OPT}
explicitly in our situation and start with the adjoint equation 
$e_y(\overline y,\overline u)^\star \overline p = -(\overline \varphi_\tau-\varphi_d)$.
This equation abbreviates a time stepping scheme, which steps backwards in time. On every time instance the adjoint Navier--Stokes
and the adjoint Cahn--Hilliard equation are sequentially coupled, where the adjoint Navier--Stokes equation can be solved independently of the adjoint Cahn--Hilliard equation.
As the adjoint equation encodes a time stepping scheme, we just state one step and note the convention, that terms with index $M+1$ are defined as zero.
We denote the adjoint velocity field by $p_v$, the adjoint phase field by $p_\varphi$ and the adjoint chemical potential as $p_\mu$. The test functions
are again denoted by $w$, $\Psi$, and $\Phi$.
On time instance $m = M,\ldots,1$ the adjoint system is given by the following time stepping scheme, backwards in time
\begin{align}
\left(\frac{\rho^m+\rho^{m-1}}{2}p_v^m - \rho^m p_v^{m+1},w^m\right)&\\
+ \tau a(\rho^m w^m,v^{m+1},p_v^{m+1})
+ \tau a(\rho^{m-1}v^{m-1}+J^m,w^m,p_v^m)&\\
+\tau (2\eta^{m}Dw^m,Dp_v^m)
-\tau(\varphi^m w^m,\nabla p_\mu^{m+1}) & = 0,\\
%
\tau a(-b\rho^\prime(\varphi^{m})\nabla \Psi^m,v^m,p_v^m)
+\tau(\varphi^{m-1}\nabla \Psi^m,p_v^m)&\\
+\frac{\tau^2}{\rho_{\min}}(|\varphi^{m-1}|^2\nabla \Psi^m,\nabla p_\mu^m)
+\tau b(\nabla \Psi^m,\nabla p_\mu^m)
-\tau(\Psi^m,p_\varphi^m) &= 0, \\
%
\frac{1}{2}(\rho^\prime(\varphi^m)\Phi^m,v^mp_v^m + v^{m+1}p_v^{m+1})
-(\rho^\prime(\varphi^m)\Phi^m v^m,p_v^{m+1})&\\
+\tau a(-b\rho^{\prime\prime}(\varphi^m)\Phi^m\nabla \mu^{m},v^{m},p_v^{m}) 
+\tau a( \rho^\prime(\varphi^m)\Phi^mv^m,v^{m+1},p_v^{m+1}) &\\
+\tau(2\eta^{\prime}(\varphi^m)\Phi^m Dv^{m},Dp_v^{m}) &\\
+\tau(\Phi^m\nabla \mu^{m+1},p_v^{m+1})
-\tau(g \rho^{\prime}(\varphi^m)\Phi^m,p_v^m) &\\
+(\Phi^m,p_\mu^m-p_\mu^{m+1})
-\tau(\Phi^mv^m,\nabla p_\mu^{m+1})
+\frac{\tau^2}{\rho_{\min}}(2\Phi^m\nabla\mu^{m+1},\nabla p_\mu^{m+1}) &\\ 
+\tau\sigma\epsilon(\nabla \Phi^m,\nabla p_\varphi^m)
+\frac{\tau\sigma}{\epsilon}(W_+^{\prime\prime}(\varphi^m)\Phi^m,p_\varphi^m)
+\frac{\tau\sigma}{\epsilon}(W_-^{\prime\prime}(\varphi^m)\Phi^m,p_\varphi^{m+1})&\\
+ r (\Phi^m,p_\varphi^m-p_\varphi^{m+1})_{\partial\Omega} & \\
+\tau \left( \frac{S_\gamma}{2}(\Phi^m,p_\varphi^m-p_\varphi^{m+1})_{\partial\Omega} + (\gamma_u^{\prime\prime}(\varphi^m)\Phi^m,p_\varphi^{m+1})_{\partial\Omega} \right) &
= -\tau(\varphi^m - \varphi_d^m,\Phi^m),
\end{align}
where $\varphi_d^m := \frac{1}{\tau}\int_{t_{m-1}}^{t_m}\varphi_d(t)\dt$.

For the optimality condition 
$\left<\alpha B^\star  B(\overline u) + e_u(\overline y,\overline u)^\star \overline p,u- \overline u\right>_{U^\star,U} \geq 0 
\quad \forall u \in U_{ad}$ we obtain the following equation
\begin{align*}
&\alpha \left<B^\star B \overline u,u-\overline u\right> _{U^\star,U}
+ \sum_{m=1}^M\left< B_m^\star(\tau \sigma_{lg}\vartheta^\prime(\varphi^{m-1})p_\varphi^m),u-\overline u\right>_{U^\star,U}\\
= &\alpha\left<B^\star B \overline u,u-\overline u\right> _{U^\star,U}
+ \left< B^\star \left(\sum_{m=1}^M \sigma_{lg}\vartheta^\prime(\varphi^{m-1})p_\varphi^m\chi_m\right),u-\overline u\right> _{U^\star,U}
 \geq 0,
\end{align*}
where $\chi_m$ denotes the characteristic function of the interval $(t_{m-1},t_m)$.

\section{Numerical Experiment: How to Make Water Run Uphill}\label{sec:N}
In this section we demonstrate the general applicability of our approach and framework.
As a test example, we consider the setup illustrated in~\cref{fig:setup_intro} 
in the introduction.
As observed in very famous experiments by~\citet{Chaudhury1992},
it is possible to push the droplet uphill against gravity by choosing a specific contact angle distribution between solid, droplet and surrounding fluid.
In addition, this example is inspired by the work reported in~\cite{Al-Sharafi2018}, 
where the heat transfer into a sliding and pinned droplet is characterised, and~\cite{tMannetje2014}, where drops are trapped due to steep changes in the contact angle.
The pinning of a sliding droplet at a specific position on a solid surface while maintaining a desired shape has interesting implications for technical applications.
At the same time, this is a challenging task.
%
Note, that in practical applications, the control patches can represent electrodes.
Again we refer to~\cite{Mugele2018} for details on the technical implementation.

\subsubsection{Implementation}
\label{ssec:N:I}
For the spatial discretization of \eqref{eq:S:1_NS}--\eqref{eq:S:4_CH2}, we consider the standard finite element concept and start with considering a subdivision of $\Omega$ into triangles.
On this triangulation we use piecewise linear and globally continuous functions for $\varphi$ and $\mu$.
To deal with the constraint of solenoidal velocity fields, we add the pressure $p$ as an additional variable. To discretize pressure and velocity, 
 we use piecewise linear and globally continuous functions for $p$ respectively piecewise quadratic and globally continuous functions for $v$ (i.e., we use Taylor--Hood elements for the solution of the Navier--Stokes part).
We refer to \cite{Bonart2019b} for more information on the implementation.
The finite elements are provided by the toolbox FEniCS~2019.1.0~\cite{fenics1,fenics_book}. 
For the solution of the arising nonlinear and linear systems and subsystems the software suite PETSc 3.8.4~\cite{petsc_webpage, petsc-user-ref, petsc-efficient} together with the direct linear solver MUMPS 5.1.1~\cite{mumps_1, mumps_2} are utilized. 
Note, that we do not apply any preconditioning or subiterations except for the Newton iterations, appearing from the nonlinearity $W_+(\varphi^m)$.
To solve the optimization problem, we use IPOPT~3.12~\cite{ipopt_webpage,ipopt} with options set to default values.

\subsection{Setup}
A single droplet (solid line in~\cref{fig:setup_intro}) is placed on an inclined surface.
If no control action is taken, the droplet slides down the surface driven by gravity.
Using given contact angle distributions, we control the advancing and receding contact angles $\theta_1$ and $\theta_2$ while the droplet slides along the surface.
In this way, we impose the desired shape and position of the droplet at any time (dashed line) 
over a time horizon $I = (0,T)$ with $T=5$.

We model this setup as follows.
In a rectangular domain $\Omega= (0,1) \times (0,0.5)$, a liquid, cap-shaped droplet $\varphi_0$ with radius $r_0=0.25$ is placed  at $m=(0.375$,$0)$ 
on a smooth, solid surface (represented by the boundary $(0,1) \times \{0\}$) with an equilibrium contact angle of $\theta_{eq}=\si{90}{\degree}$. 
The explicit definition of $\varphi_0$ is stated below.
The inclination angle of the plate to the horizontal is $\alpha = \si{-15}{\degree}$, which is modelled by a corresponding inclination angle of the gravitational force.

We have no-slip boundary conditions for the velocity on the left and right side and free-slip on the top side.
The condition~\cref{eq:M:7_CH_BC} is applied on the bottom boundary, i.e. $(0,1) \times \{0\}$.

\Cref{tab:setup} lists the applied parameter values (the values are taken from~\cite{Hysing_Turek_quantitative_benchmark_computations_of_two_dimensional_bubble_dynamics,Bonart2019b})
for the physical and numerical parameters in model \eqref{eq:M:1_NS1}--\eqref{eq:M:8_mu_neumann}.
The density of the droplet is greater than the density of the surrounding fluid and
since the inclination is negative gravity pulls the droplet down.

%

\begin{table}[h!]
\centering
\begin{tabular}{cccccccc|cc}
		$\sigma_{lg}$	&$\rho_l$	&$\rho_g$	&$\eta_l$	&$\eta_g$	&$g$ &$\alpha$ &$r$& $\epsilon$		&$b$		\\
\cmidrule(r){1-10}
		24.5			&1000		&100		&10			&1			&0.98 &\SI{-0.15}{\degree}&0.35& $\num[scientific-notation = true]{2e-2}$&$\num[scientific-notation = true]{2e-5}$\\
\end{tabular}
\caption{Physical and numerical parameters of droplet, surrounding fluid and solid surface.}
\label{tab:setup}
\end{table}

Note, that $\sigma_{lg}$ denotes the physical surface tension that has to be scaled by the constant $c_W$ in \eqref{eq:M:4_CH2} that depends on the free energy potential $W$.
Here we use 
\begin{align}
 W(\varphi) :=
 \begin{cases}
  \frac{1}{4}(1-\varphi^2)^2 & \mbox{if } |\varphi| \leq 1,\\
  (|\varphi|-1)^2 & \mbox{if } |\varphi| > 1,
 \end{cases}
 \label{eq:M:W}
\end{align} 
together with  the convex-concave splitting
\begin{align}
 W_+(\varphi) :=
 \begin{cases}
\frac{1}{4}\varphi^4-\frac{1}{4} & \mbox{if } |\varphi| \leq 1,\\
\frac{1}{2}(3\varphi^2 - 4|\varphi| + 1) & \mbox{if } |\varphi| > 1,
 \end{cases}
 &&
 W_-(\varphi) :=
 \frac{1}{2}(1-\varphi^2)
 \label{eq:M:W_split}
\end{align}
This choice leads to  $c_W = \frac{3}{2\sqrt 2}$ and $\Phi_0(z) = \tanh(z/\sqrt{2})$, see \Cref{rm:M:Potential} and \cite[Rem.~2]{Bonart2019b}.
Using $\Phi_0$ we define $\varphi_0(x) := \Phi_0( (\|x-m\| - r_0)/\epsilon)$.

Exemplarily, we use the dashed line in~\cref{fig:results} as the desired and constant shape $\varphi_d$.
It is created by placing a liquid, cap-shaped droplet with radius $r_d=0.25$ at $m_d=(0.625$,$0)$ and initial contact angle of \SI{90}{\degree} on a perfectly horizontal surface (no inclination with $\alpha=\SI{0}{\degree}$).
Thus we have $\varphi_d(x) = \Phi_0( (\|x-m_d\| - r_d)/\epsilon)$.
Then the static contact angle is set to $\theta_{eq} =\SI{135}{\degree}$ 
and the forward model is simulated until the droplet reaches its equilibrium shape.
In absence of any inclination the droplet does not slide along the surface.
Note, that setting $\theta_{eq} = \SI{135}{\degree}$ is equivalent to simulating with the constant control $Bu = Bu + \cos(\SI{90}{\degree}) = \cos(\SI{135}{\degree})$.
The bounds for the optimizer are set to  $\cos(\theta_{\min})=-0.9$ and $\cos(\theta_{\max}) = 0.9$.

\subsection{The control space $U$}
We model the control by a linear combination of fixed controls.
Subsequently, we optimize the coefficients in this linear combination, compare~\cite{GarHK_optContr_twoPhaseFlow}.
These fixed controls either depend on time or on space.
We note, that we can describe controls that only depend on time, controls that only depend on space and controls that depend on both space and time by such a model. 

Let $g_r \in L^2(0,T)$, $r=1,\ldots,R$, denote $R$ given time depending functions, and let $f_s \in L^2(\partial\Omega)$, $s=1,\ldots,S$, denote $S$ given spatially depending functions.
We define $U = \mathbb{R}^{RS}$ and  $B : U\to L^2(0,T; L^2(\partial\Omega))$ by
\begin{align*}
Bu = \sum_{r=1}^R\sum_{s=1}^S u_{rs}g_r(t)f_s(x),
\end{align*}
where $u_{rs}$ denote the entries of $u\in \mathbb R^{RS} $
in a suitable ordering.
The adjoint operator $B^\star : L^2(0,T;L^2(\partial\Omega)) \to U$ is for
an arbitrary $p \in L^2(0,T;L^2(\partial\Omega))$ given by
\begin{align*}
(B^\star p)_{rs} = (p,g_r(t)f_s(x))_{L^2(0,T; L^2(\partial\Omega))}. 
\end{align*}

Note that we can set $R=1$ and $g_1(t) = 1$ to model controls, that only depend on space
and $S=1$ and $f_s(x) = 1$ to model controls, that only depend on time.

In the following we use $S=10$, i.e., ten equally sized control patches, and $R=5$ to allow the control to switch at five equally distributed instances in time.
Thus we use $f_s(x) = \chi_{\left(\frac{1}{10}s,\frac{1}{10}(s+1)\right)}(x)$, $s=0,\ldots,9$
and $g_r(t) = \chi_{(r,r+1)}(t)$, $r=0,\ldots,4$,
where $\chi_J$ denotes the characteristic function of the interval $J$.

\subsection{Results}
In \Cref{fig:results} we show numerical results that indicate that our approach is successful in the present situation. We present phase fields by their zero-level isoline.
At first we demonstrate what happens to the droplet on the inclined surface without any control at all ($Bu=0$ for all times), see the first column in~\cref{fig:results}.
Secondly, one might naively set the control to the finally desired contact angle of \SI{135}{\degree}.
This is displayed in the second column.
As expected in both cases the droplet starts running downhill.
Subsequently, the droplets are far from the desired shape as well as the desired position at the specified time.
Certainly, finding the right control actions by trial-and-error is cumbersome.
Furthermore, we are not interested in gaining any control but one control that is optimal with respect to the cost functional.

Finally, in the third column of~\cref{fig:results} the droplet's evolution is shown, if an optimal control is applied.
We notice, that the droplet runs uphill and it matches the desired shape at the desired position very well for $t=5$.
Comparing this to the naive approaches with $Bu=0$ and $Bu=-0.7071$ the improvement is tremendous.
During the sliding and pinning the droplet exhibits multiple unusual shapes due to the strong impact of the control actions on the droplet.
The control actions $Bu$ in each time interval are shown in the last column of~\cref{fig:results}.
It is striking how complex these optimal controls $Bu$ are.
While pulled by gravity, the droplet's contact points are forced to recede and spread multiple times.
See for example the stripe between 0.8 to 0.9 at the time intervals 3 to 4 and 4 to 5. 
The control action jumps from large positive values (spreading, so the droplet moves uphill) to negative values (receding, so the droplet meets the desired shape).
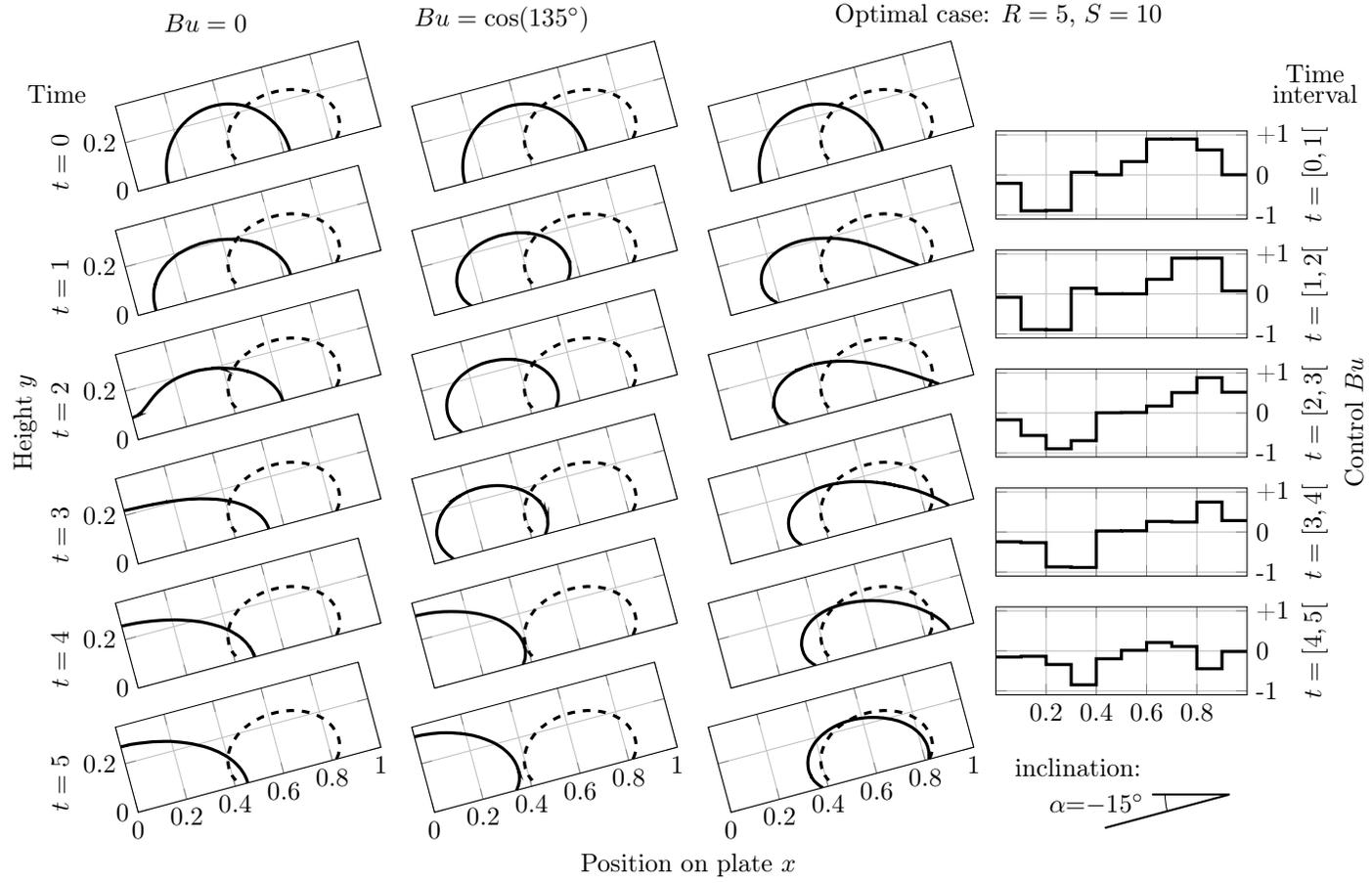
\begin{sidewaysfigure}[p]
\tikzsetnextfilename{results}
\centering
\begin{tikzpicture}
	\centering
	\begin{groupplot}[
			group style = {
				group size = 3 by 6
				,vertical sep=-10pt,horizontal sep=12pt
				,ylabels at=edge left
				,xlabels at=edge bottom
               	,yticklabels at=edge left 
				,xticklabels at=edge bottom
				, group name=left plots
			},
    		width = 5cm
			,xlabel near ticks
			,ylabel near ticks
			,unit vector ratio=1 1 1
			, xmin=0, xmax=1.0
			, ymin=0, ymax=0.35
			, ytick={0, 0.2}
			, grid=both
			, ylabel shift = -0 pt
			, title style={yshift=-15pt}
			]
			\nextgroupplot[title={$Bu=0$}, ylabel={$t=0$}, rotate around={15:(current axis.origin)}] 
				\addplot[very thick, smooth] plot file{data/forward_zero/isolines/isoline_phi0_tau0.0.dat};
				\addplot[very thick, smooth, dashed] plot file{data/isoline_phid.dat};
			\nextgroupplot[title={$Bu=\cos(\SI{135}{\degree})$}, rotate around={15:(current axis.origin)}] 
				\addplot[very thick, smooth] plot file{data/forward_135/isolines/isoline_phi0_tau0.0.dat};
				\addplot[very thick, smooth, dashed] plot file{data/isoline_phid.dat};
			\nextgroupplot[, rotate around={15:(current axis.origin)}] 
				\addplot[very thick, smooth] plot file{data/forward_opt/isolines/isoline_phi0_tau0.0.dat};
				\addplot[very thick, smooth, dashed] plot file{data/isoline_phid.dat};
			\nextgroupplot[ylabel={$t=1$}, rotate around={15:(current axis.origin)}] 
				\addplot[very thick, smooth] plot file{data/forward_zero/isolines/isoline_phi0_tau1.0.dat};
				\addplot[very thick, smooth, dashed] plot file{data/isoline_phid.dat};
			\nextgroupplot[, rotate around={15:(current axis.origin)}] 
				\addplot[very thick, smooth] plot file{data/forward_135/isolines/isoline_phi0_tau1.0.dat};
				\addplot[very thick, smooth, dashed] plot file{data/isoline_phid.dat};
			\nextgroupplot[, rotate around={15:(current axis.origin)}] 
				\addplot[very thick, smooth] plot file{data/forward_opt/isolines/isoline_phi0_tau1.0.dat};
				\addplot[very thick, smooth, dashed] plot file{data/isoline_phid.dat};
			\nextgroupplot[ylabel={$t=2$}, rotate around={15:(current axis.origin)}] 
				\addplot[very thick, smooth] plot file{data/forward_zero/isolines/isoline_phi0_tau2.0.dat};
				\addplot[very thick, smooth, dashed] plot file{data/isoline_phid.dat};
			\nextgroupplot[, rotate around={15:(current axis.origin)}] 
				\addplot[very thick, smooth] plot file{data/forward_135/isolines/isoline_phi0_tau2.0.dat};
				\addplot[very thick, smooth, dashed] plot file{data/isoline_phid.dat};
			\nextgroupplot[, rotate around={15:(current axis.origin)}] 
				\addplot[very thick, smooth] plot file{data/forward_opt/isolines/isoline_phi0_tau2.0.dat};
				\addplot[very thick, smooth, dashed] plot file{data/isoline_phid.dat};
			\nextgroupplot[ylabel={$t=3$}, rotate around={15:(current axis.origin)}] 
				\addplot[very thick, smooth] plot file{data/forward_zero/isolines/isoline_phi0_tau3.0.dat};
				\addplot[very thick, smooth, dashed] plot file{data/isoline_phid.dat};
			\nextgroupplot[, rotate around={15:(current axis.origin)}] 
				\addplot[very thick, smooth] plot file{data/forward_135/isolines/isoline_phi0_tau3.0.dat};
				\addplot[very thick, smooth, dashed] plot file{data/isoline_phid.dat};
			\nextgroupplot[, rotate around={15:(current axis.origin)}] 
				\addplot[very thick, smooth] plot file{data/forward_opt/isolines/isoline_phi0_tau3.0.dat};
				\addplot[very thick, smooth, dashed] plot file{data/isoline_phid.dat};
			\nextgroupplot[ylabel={$t=4$}, rotate around={15:(current axis.origin)}] 
				\addplot[very thick, smooth] plot file{data/forward_zero/isolines/isoline_phi0_tau4.0.dat};
				\addplot[very thick, smooth, dashed] plot file{data/isoline_phid.dat};
			\nextgroupplot[, rotate around={15:(current axis.origin)}] 
				\addplot[very thick, smooth] plot file{data/forward_135/isolines/isoline_phi0_tau4.0.dat};
				\addplot[very thick, smooth, dashed] plot file{data/isoline_phid.dat};
			\nextgroupplot[, rotate around={15:(current axis.origin)}] 
				\addplot[very thick, smooth] plot file{data/forward_opt/isolines/isoline_phi0_tau4.0.dat};
				\addplot[very thick, smooth, dashed] plot file{data/isoline_phid.dat};
			\nextgroupplot[ylabel={$t=5$}, rotate around={15:(current axis.origin)}] 
				\addplot[very thick, smooth] plot file{data/forward_zero/isolines/isoline_phi0_tau5.0.dat};
				\addplot[very thick, smooth, dashed] plot file{data/isoline_phid.dat};
			\nextgroupplot[, rotate around={15:(current axis.origin)}] 
				\addplot[very thick, smooth] plot file{data/forward_135/isolines/isoline_phi0_tau5.0.dat};
				\addplot[very thick, smooth, dashed] plot file{data/isoline_phid.dat};
			\nextgroupplot[, rotate around={15:(current axis.origin)}] 
				\addplot[very thick, smooth] plot file{data/forward_opt/isolines/isoline_phi0_tau5.0.dat};
				\addplot[very thick, smooth, dashed] plot file{data/isoline_phid.dat};
		\end{groupplot}
		\begin{groupplot}[
			group style = {
				,group size = 1 by 5
				,vertical sep=12pt,horizontal sep=12pt
				,ylabels at=edge right
				,xlabels at=edge bottom
               	,yticklabels at=edge right
				,xticklabels at=edge bottom
			},
    		width = 5cm, height = 2.778cm
			,xlabel near ticks
			,ylabel near ticks
			, xmin=0, xmax=1.0
			, ymin=-1.1, ymax=1.1
			, ytick={-1, 0,+1}
			, yticklabels={-1, 0, +1}
			, xtick={0.2,0.4,0.6,0.8}
			, grid=both
			, ylabel shift = -2 pt
			, title style={yshift=-7pt}
			, every axis xlabel/.style={}
			]
			\nextgroupplot[anchor=north west, at={($(left plots c3r1.east) + (0.3cm,-0.2cm)$)}, ylabel={$t=[0,1[$}] 
				\addplot[very thick, const plot mark left]  table [x index=0, y index=1]{data/forward_opt/controls/control_opt_0.dat};
			\nextgroupplot[ylabel={$t=[1,2[$}] 
				\addplot[very thick, const plot mark left]  table [x index=0, y index=1]{data/forward_opt/controls/control_opt_1.dat};
			\nextgroupplot[ylabel={$t=[2,3[$}] 
				\addplot[very thick, const plot mark left]  table [x index=0, y index=1]{data/forward_opt/controls/control_opt_2.dat};
			\nextgroupplot[ylabel={$t=[3,4[$}] 
				\addplot[very thick, const plot mark left]  table [x index=0, y index=1]{data/forward_opt/controls/control_opt_3.dat};
			\nextgroupplot[ylabel={$t=[4,5[$}] 
				\addplot[very thick, const plot mark left]  table [x index=0, y index=1]{data/forward_opt/controls/control_opt_4.dat};
			\end{groupplot}
		\node[rotate=90,anchor=center] at ($(current bounding box.west) + (-7pt,0)$) {Height $y$};
		\node[rotate=90,anchor=center] at ($(current bounding box.east) + (+7pt,0)$) {Control $Bu$};
		\node[anchor=center] at ($(current bounding box.south) + (0pt,-7pt)$) {Position on plate $x$};
		\node[above] at ($(current bounding box.north) + (120pt,-15pt)$) {Optimal case: $R=5$, $S=10$};
		\node[above] at ($(current bounding box.north west) + (20pt,-45pt)$) {Time}; 
		\node[above,align=center, execute at begin node=\setlength{\baselineskip}{0pt}] at ($(current bounding box.north east) + (-22pt,-45pt)$) {Time\\interval};
		\coordinate (origo) at ($(current bounding box.south east) + (-55pt,35pt)$);
		\draw[thick] (origo) -- ++(-30pt,0) node (mary) []{};
		\node[above left = 0pt and -3pt of mary] {inclination:};
		\draw[thick] (origo) -- ++(-165:50pt) node (bob) []{};
		\pic [draw, left, "$\alpha{=}{-}\SI{15}{\degree}$", angle eccentricity=1.2, angle radius=25pt] {angle = mary--origo--bob};
    {angle=a--b--c};
	\end{tikzpicture} 
	\caption{Development of the sliding and pinned droplets over time on a plate with an inclination to the horizontal of $\alpha=\SI{-15}{\degree}$: 
	no control action with $Bu=0$ (first column), 
	constant control action with $Bu=\cos(\SI{135}{\degree})$ (second column) 
	and optimal control action for $R=5$ and $S=10$ (third column).
	In the fourth column the optimal control action is displayed for the time intervals. The desired shape and position of the droplet $\varphi_d$ is included as its zero level set as dashed line.
	}	
	\label{fig:results}
\end{sidewaysfigure}

\section{Conclusion}
In this work we considered an optimal control problem for the shape and position of droplets sliding on solid surfaces.
Based on our studies in~\cite{Bonart2019b} on numerical schemes for two-phase flows involving moving contact line dynamics, we chose a detailed phase field model as the physical constraint.
We showed higher regularity for the unique solution to this time discretization scheme for this highly nonlinear system and were able to proof existence of solutions to the corresponding optimal control problem. 
Further we derived first order optimality conditions that we used in the quasi-Newton algorithm of the interior-point solver IPOPT.

To demonstrate our approach we considered the active control of a sliding droplet using temporally and spatially varying contact angle distributions.
In this basic example common for droplet-based microfluidics the droplet slides on an inclined surface uphill against gravity.
The final droplet matched the desired shape and position almost perfectly.
To our knowledge the simultaneous control of the optimal shape and position of sliding droplets on solid surfaces has not been previously demonstrated.
Our work indicates that the active control of contact angles is a powerful approach towards controlled transport of droplets in microfluidic applications.
In future work, we investigate the particular optimization problem and the whole control process and incorporate additional
constraints that might stem from manufacturing or practical application.

\section*{Acknowledgments}

We thank Constantin Christof and Johannes Pfefferer for very fruitful discussions 
on the regularity theory and convergence theory used in Section~\ref{sec:S}.

\bibliographystyle{apalike}
\bibliography{literature}

\end{document}